\documentclass[11pt]{article}

\usepackage{amsmath,amsthm}

\usepackage{amssymb,latexsym}

\usepackage{enumerate}

\newcommand{\BB}{{\cal B}}

\newcommand{\EE}{{\cal E}}
\newcommand{\FF}{{\cal F}}

\newcommand{\HH}{{\cal H}}

\newcommand{\MM}{{\cal M}}

\newcommand{\PP}{{\cal P}}

\newcommand{\VV}{{\cal V}}
\newcommand{\WW}{{\cal W}}

\newcommand{\BN}{{\mathbb N}}
\newcommand{\BR}{{\mathbb R}}
\newcommand{\BX}{{\mathbb X}}

\newcommand{\dyw}{\mbox{\rm div}}
\newcommand{\fch}{{\mathbf{1}}}

\newtheorem{theorem}{\bf Theorem}[section]
\newtheorem{proposition}[theorem]{\bf Proposition}
\newtheorem{lemma}[theorem]{\bf Lemma}%[subsection]
\newtheorem{corollary}[theorem]{\bf Corollary}

\theoremstyle{definition}
\newtheorem*{definition}{Definition}
\newtheorem{example}[theorem]{\bf Example}

\newtheorem{remark}[theorem]{Remark}

\numberwithin{equation}{section}

\begin{document}

\title {Smooth measures and capacities associated with nonlocal parabolic operators}
%On the decomposition of soft measures in the parabolic case
\author {Tomasz Klimsiak and Andrzej Rozkosz\\
{\small Faculty of Mathematics and Computer Science,
Nicolaus Copernicus University} \\
{\small  Chopina 12/18, 87--100 Toru\'n, Poland}\\
{\small E-mail addresses: tomas@mat.umk.pl (T. Klimsiak), rozkosz@mat.umk.pl (A. Rozkosz)}}
\date{}
\maketitle
\begin{abstract}
%We consider the time dependent Dirichlet form $\EE$ associated with a family of regular
%(non-symmetric) Dirichlet forms having common domain and  satisfying some mild regularity  %assumptions. We provide a decomposition  of smooth measures  with
%respect to the capacity associated with $\EE$. We apply this decomposition to
%the study of the structure of additive functionals in the Revuz correspondence with diffuse %measures.

We consider a family  $\{L_t,\, t\in [0,T]\}$ of closed operators generated by a family of regular (non-symmetric) Dirichlet forms $\{(B^{(t)},V),t\in[0,T]\}$ on $L^2(E;m)$.
We show that a bounded (signed) measure $\mu$ on $(0,T)\times E$ is smooth, i.e. charges no set of zero parabolic capacity associated with $\frac{\partial}{\partial t}+L_t$, if and only if $\mu$ is of the form $\mu=f\cdot m_1+g_1+\partial_tg_2$
with $f\in L^1((0,T)\times E;dt\otimes m)$, $g_1\in L^2(0,T;V')$, $g_2\in L^2(0,T;V)$.
We apply this decomposition to the study of the structure of additive functionals in the Revuz correspondence with smooth measures. As  a by-product,
we also give some existence and uniqueness results for solutions of semilinear  equations involving the operator $\frac{\partial}{\partial t}+L_t$ and a functional from the dual $\WW'$ of the space $\WW=\{u\in L^2(0,T;V):\partial_t u\in L^2(0,T;V')\}$ on the right-hand side of the equation.
\end{abstract}

%\footnotetext{T. Klimsiak and Andrzej Rozkosz:
%Faculty of Mathematics and Computer Science, Nicolaus Copernicus
%University, Chopina 12/18, 87-100 Toru\'n, Poland; e-mail addresses:
%tomas@mat.umk.pl. (T. Klimsiak), rozkosz@mat.umk.pl. (A. Rozkosz).}



%\footnotetext{This work was supported by Polish National Science Centre
%(grant no. 2016/23/B/ST1/01543).}

\section{Introduction}

In the study of parabolic problems of the form
\begin{equation}
\label{eq1.3b} -\partial_tu-\Delta_pu=f(\cdot,u)+\mu\quad\mbox{in }D,\quad u|_{(0,T)\times\partial D}=0,\quad u(T,\cdot)=\varphi,
\end{equation}
where $D$ is a bounded open set in $\BR^d$,  $\Delta_p$ is the usual $p$-Laplacian, $p>1$, and $\mu$  is a bounded measure on $D_{0,T}:=(0,T)\times D$ charging no set of zero parabolic $p$-capacity associated with $\frac{\partial}{\partial t}-\Delta_p$ (see below) an important role is played by the result on the decomposition  of $\mu$ proved by Droniou, Porretta and Prignet \cite{DPP} (see,
e.g.,  \cite{DPP,Pe,PPP2}; note that in these papers more general than $\Delta_p$ operators of the form $A(u)=\dyw\, a(t,x,\nabla u)$ are considered).
The decomposition proved in \cite{DPP} says that each such measure $\mu$ (we call it diffuse)
is of the form
\begin{equation}
\label{eq1.4c}
\mu=f+\mbox{div}(G)+\partial_tg
\end{equation}
for some  $f\in L^1(D_{0,T})$, $G\in(L^{p'}(D_{0,T}))^d$ and $g\in L^p(0,T;W^{1,p}_0(D)\cap L^2(D))$. Recently, in \cite{KR}, the  converse to this result was proved.
%It says that if $p>(2d+1)/(d+1)$, then  each bounded Borel measure $\mu$ on $D_{0,T}$ %admitting decomposition (\ref{eq1.4c}) is diffuse.
The decomposition (\ref{eq1.4c}) is a counterpart to the decomposition of diffuse measures proved in the stationary case  by Boccardo, Gallou\"et and Orsina \cite{BGO}.
The decomposition of \cite{BGO} was extended to the Dirichlet forms setting in \cite{KR:BPAN}.

There has recently been increasing interest in semilinear
evolution problems of the form
\begin{equation}
\label{eq1.2a}
-\frac{\partial u}{\partial t}-L_tu=f(\cdot,u)+\mu,\qquad u(T,\cdot)=\varphi,
\end{equation}
involving  operators $L_t$  associated with
a (possibly nonlocal)  Dirichlet form and  bounded measure that do not charge the sets of zero parabolic capacity associated with $\partial _t+L_t$ (see \cite{K:JEE,K:JFA,KR:NoD} and the references therein).  Motivated by possible applications to problems of the form (\ref{eq1.2a}), in the present paper we investigate the structure of such measures.  We extend the results of \cite{KR:BPAN} to the parabolic setting
and at the same time the results of \cite{DPP} with $p=2$ to more general parabolic operators.
As a by-product, we obtain some results on the existence of solutions to equations  of the form (\ref{eq1.2a}) with $\mu\in\WW'_0$ and on the structure of additive functionals associated in the Revuz sense with bounded smooth measures.

Let  $E$ be a locally compact separable metric space,  $E_{0,T}:=(0,T)\times E$ for some $T>0$, and let $m$ be a
Radon measure on $E$ such that $\mbox{supp}[m]=E$. In the paper we consider
%the  generalized   Dirichlet form $\EE$ associated with
smooth measures with respect to  parabolic capacities associated with a family $\{L_t,\, t\in [0,T]\}$ of closed operators generated by a family $\{(B^{(t)},V),t\in[0,T]\}$ of regular
(non-symmetric) Dirichlet forms on $L^2(E;m)$, with common domain
$V$, satisfying some mild regularity  assumptions. Our general Dirichlet forms setting allows us to treat both local and nonlocal operators. 
The results of the paper are new even for local operators. However, in our opinion, the most interesting fact is that we are able to describe the structure of smooth measures (and related additive functionals) for capacities associated with quite large class of parabolic nonlocal operators.

The model example of the family   of local operators
satisfying our assumptions is the family of divergence form operators
\begin{equation}
\label{eq1.2}
L_t=\frac12\sum^d_{i,j=1}\frac{\partial}{\partial x_i}
\Big(a_{ij}(t,x)\frac{\partial}{\partial x_j}\Big),\quad t\in[0,T],\quad x\in D,
\end{equation}
on $L^2(D;dx)$ with zero Dirichlet boundary conditions. In (\ref{eq1.2}), $D$ is a bounded open subset of $\BR^d$ and $\{a_{ij}(t,x)\}_{i,j=1,\dots,d}$ is a symmetric  uniformly elliptic matrix with bounded measurable elements.
In this case the family
$\{(B^{(t)},V),t\in[0,T]\}$ is of the form
\begin{equation}
\label{eq1.1}
B^{(t)}(u,v)=\sum^d_{i,j=1}\int_Da_{ij}(t,x)\frac{\partial u}{\partial x_i}(x)\frac{\partial v}{\partial x_j}(x)\,dx,\quad u, v\in V:=H^1_0(D).
\end{equation}
A model example of the family of  nonlocal operators satisfying our assumptions is  the family consisting of single  fractional Laplace operator
\begin{equation}
\label{eq1.3}
L_t=-(-\Delta)^{\alpha/2}
\end{equation}
on $L^2(D;dx)$ with zero exterior condition. Here $\alpha\in(0,2)$ and $D$ is an open subset of $\BR^d$. In this case $(B^{(t)},V)=(B,V)$, $t\in[0,T]$, where
\[
B(u,v)=\int_{\BR^d}\hat u(x)\bar{\hat v}(x)|x|^\alpha\,dx,\quad u, v\in V:=\{w\in L^2(D;dx):\int_{\BR^d}|\hat u|^2|x|^\alpha\,dx<\infty \}
\]
and $\hat u$ (resp.  $\hat v$) is the Fourier transform of $u$ (resp. $v$) (see \cite[Example 1.4.1, Theorem 4.4.3]{FOT}).

In case of problem (\ref{eq1.3b}), the natural capacity is  the $p$-parabolic capacity $\mbox{cap}_p$ defined  for open set $U\subset D_{0,T}$ by
\[
\mbox{cap}_p(U)=\inf\{\|\partial_tu\|_{L^{p'}(0,T;V'_p)}
+\|u\|_{L^p(0,T;V_p)}:u\ge \mathbf{1}_{U}\,\,
dt\otimes dx\mbox{-a.e.}\},
\]
where $V_p=W^{1,p}_0\cap L^2(D)$ and $V'_p$ is the dual of $V_p$ (see \cite{DPP,Pierre3}).
To study evolution problems with the operator $\partial _t+L_t$, Pierre \cite{Pierre3} introduced the capacity  $\mbox{c}_2$ defined for open $U\subset E_{0,T}$ by
\[
c_2(U)=\inf\{\|\partial_tu\|_{L^{2}(0,T;V')} +\Big(\int_0^TB^{(t)}(u(t),u(t))\,dt\Big)^{1/2}:u\ge
\mathbf{1}_{U}\,\,m_1\mbox{-a.e.}\},
\]
where $m_1=dt\otimes m$. In the potential theory, Borel measures on $E_{0,T}$ which do not charge sets of zero capacity $\mbox{c}_2$ and satisfy some quasi-finitness condition are called  smooth measures. In particular, each  bounded Borel measure ``absolutely continuous" with respect to c$_2$ is smooth. Of course, in  case of operators of the form (\ref{eq1.2})
the classes of diffuse measures and bounded smooth measures coincide.

In our main theorem  we extend (\ref{eq1.4c})  to the Dirichlet form setting.  Let  $\MM_{0,b}(E_{0,T})$ denote the set of all bounded smooth measures on $E_{0,T}$.
We show that each $\mu\in\MM_{0,b}(E_{0,T})$ admits decomposition of the form
\begin{equation}
\label{eq1.4} \mu=f\cdot m_1+g_1+\partial_tg_2
\end{equation}
with $f\in L^1(E_{0,T};m_1)$, $g_1\in\VV'=L^2(0,T;V')$, $g_2\in\VV=L^2(0,T;V)$, i.e. for every bounded
quasi-continuous $\eta\in\WW_0=\{u\in\VV:\partial_tu\in\VV',u(0)=0\}$, we have
\[
\int_{E_{0,T}}\eta\,d\mu=\int_{E_{0,T}}f\,dt\,dm+\langle
g_1,\eta\rangle-\langle\partial_t\eta,g_2\rangle,
\]
where $\langle\cdot,\cdot\rangle$ denotes the duality between
$\VV'$ and $\VV$. We also show the converse of this theorem. Namely,  each $\mu\in\MM_b(E_{0,T})$  having decomposition (\ref{eq1.4})  is smooth.
Note that the converse is an extension of the result proved by
Fukushima \cite{Fukushima}  to time dependent Dirichlet forms. The proof of the fact that  $\mu\in\MM_{0,b}(E_{0,T})$ can be written in the form (\ref{eq1.4}) is purely analytic. Essential to the proof of the converse part are  probabilistic methods.

In applications to (\ref{eq1.3b}), the analysis of additional  properties
of the term $g$ appearing in the decomposition (\ref{eq1.4c}) proved to be important.
For instance, crucial to the definition  and the existence result of a solution $u$ of  (\ref{eq1.3b}) is the fact that $g$ regularizes $u$ with respect to time in the sense that $u-g\in W\subset C([0,T],L^p(D))$, where $W=\{u\in L^p(0,T;V_p):\partial_t u\in L^{p'}(0,T;V'_p)$.
This property together with some other useful properties of $g$ has been proved in \cite{DPP}. In the present paper,  applying the potential theory tools, we show that $g_2$ from (\ref{eq1.4}) enjoys similar properties. We also show some new results on the regularity of $g_2$.  We show the following useful properties.
\begin{enumerate}[{-}]
\item $g_2$ has an $m_1$-version $\tilde g_2$ which is quasi-c\`adl\`ag (i.e quasi-right-continuous with left limits; this notion generalizes the notion of qusi-continuity; see Section \ref{sec8}) and $g_2$ is a difference of c$_2$-quasi-l.s.c. functions.
\item  $g_2$ has an $m_1$-version $\tilde g_2$ which
is c$_2$-quasi-bounded, i.e.  there exists an increasing sequence $\{F_n\}$  of closed subsets of $E_{0,T}$  such that
$c_2(E_{0,T}\setminus F_n)\rightarrow 0$ as $n\rightarrow\infty$ and
$\|\tilde g_2\mathbf{1}_{F_n}\|_\infty<\infty$, $n\ge1$.

\item The function $[0,T]\ni t\mapsto g_2(t)\in L^2(E;m)$ is c\`adl\`ag (right-continuous with left limits) and $g_2(T-):=\lim_{t\rightarrow T^-}g_2(t)=0$.

\item The measure $\mu_t$ defined by $\mu_t(B)=\mu(\{t\}\times B)$  for Borel sets $B\subset E$ is absolutely continuous with respect to $m$ and $\mu_t=(g_2(t)-g_2(t-))\cdot m$.

\item  If $\mu$ admits decomposition $\mu=f'\cdot m_1+g'_1+\partial_tg'_2$ with some
$f'\in L^1(E_{0,T};m_1)$, $g'_1\in\VV'$, $g'_2\in\VV$,
then the function
$[0,T]\ni t\mapsto (g_2-g_2')(t)\in L^2(E;m)$
belongs to $C([0,T],L^1(E;\rho\cdot m))$ for every positive  Borel function $\rho$ on $E$   such that $\int_E\rho\,dm<\infty$.
\end{enumerate}

It is worth pointing out here that the proofs of the above results in the general setting requires us to use quite  different methods then those  used in \cite{DPP} for the Leray-Lions type operators,  which are strongly based on the regularization of the measure $\mu$ by a convolution operator. %which  is contractive in this special case with respect to semigroup generated by

%In the stationary case  it is also known  that  if $\mu$ is a bounded Borel measure on $E$ %(resp. $D$) admitting decomposition (\ref{eq1.4}) (resp. (\ref{eq1.4c})), of course then with %time independent $f,g_1$ and $g_2\equiv 0$
%(resp. $f,G$ and $g\equiv 0$),    then it is a smooth measure (resp. diffuse measure; see %\cite{BGO,KR:BPAN}). In the evolution  setting only a partial result in this direction was %known. Let $\MM_b(E_{0,T})$ (resp. $\MM_b(D_{0,T})$) denote the set of all bounded (signed) %Borel measures on $E_{0,T}$ (resp. $D_{0,T}$). In Petitta, Ponce and Porretta \cite{PPP2} %(see also \cite{PPP1}) it was shown that  if $\mu\in\MM_b(D_{0,T})$ admits  decomposition %(\ref{eq1.4c}) it is diffuse if   one assume additionally that $g\in L^{\infty}(D_{0,T})$.
%The problem whether one can dispense with the additional assumption that $g\in %L^{\infty}(D_{0,T})$ was left  open.

In the proof of our main decomposition theorem, we apply some deep results from the potential theory for evolution operators proved  by Pierre \cite{Pierre1,Pierre2,Pierre3}, as well as from the probabilistic potential theory for time dependent or generalized Dirichlet forms developed in the papers by Oshima \cite{O1,O2,O3,O4}, Stannat \cite{St1,St2} and Trutnau \cite{T1,T2}. In these papers the definitions of the capacity (and hence some quasi-notions) are different. In Section \ref{sec3}, which is technical but important for us, we show that all these capacities are in fact equivalent on $E_{0,T}$. This allows us to
apply freely the results from the papers mentioned above.

%M. Pierre considered
%(one of the equivalent capacities, see \cite{Pierre3}) capacity c$_2$ defined by %(\ref{eq5.e}). In the papers of Stannat and Trutnau  it is considered
%the framework of generalized Dirichlet forms and naturally associated capacity Cap$_\psi$,
%\[
%\mbox{Cap}_{\psi}(U)=((G_1\psi)_U,\psi)_{\HH},
%\]
%for relatively compact open sets $U\subset E_{0,T}$, where $\{G_\alpha,\,\alpha>0\}$ is the %resolvent of $\LL$, $(G_1\psi)_U$ is the reduced function and $\psi$
%arbitrary strictly positive Borel function on  $E_{0,T}$ such that %$\int_{E_{0,T}}\psi\,dm_1<\infty$, $\psi\le 1$ (see Section 2). Y. Oshima considered the %framework of time-dependent Dirichlet forms and naturally associated capacity
%Cap$^1$ on $\BR\times E$,
%\[
%Cap^1(U)=\hat\mu_U(\overline U),
%\]
%for relatively compact open sets $U\subset \BR\times E$, where $\hat\mu_U$ is the so called %co-equilibrium measure for $U$. In Section 2 we spent a time to show that all this capacities %are equivalent on $E_{0,T}$, which allows us to
%apply freely all the mention result and approaches  of potential theory.

One of the most important ingredient of the proof that   each $\mu\in\MM_b(E_{0,T})$  having decomposition (\ref{eq1.4})  is smooth  is an existence result
for the Cauchy problem (\ref{eq1.2a}) with $\mu\in\WW'_0$ and $f$ not depending on $u$. If $\mu\in L^2(0,T;V')$, then the existence of a solution to (\ref{eq1.2a}) follows
from the classical theory of variational inequalities (see \cite{Lions}). However, if $\mu\in\WW'_0$, then the situation is more difficult. To prove the existence of a solution, a decomposition similar to (\ref{eq1.4}),  but for functionals from $\WW'_0$ is needed.  In case of  (\ref{eq1.3b}),  such a decomposition and an existence result were proved in \cite{DPP}.
In Section \ref{sec5}, we prove a similar decomposition result for $\mu\in\WW'_0$, and then, in Section  \ref{sec6}, we deal with the existence of a solution of (\ref{eq1.2a}) with $\mu\in\WW'_0$. Though for our applications
we only need  the existence result for liner equations, in the paper we show that there exist a solution for semiliner problem (\ref{eq1.2a}) under the assumption that
$\varphi\in L^2(E;m)$ and $u\mapsto f(\cdot,u)$ is continuous, nonincreasing and  satisfies the linear growth condition. We think that this result may be of independent interest.
Finally, note here that in the proof that $\mu$
given by (\ref{eq1.4}) is smooth we use  a very recent result from the paper by Beznea and  C\^impean \cite{BC} on the characterization of quasimartingale functions.

It is well known that there is a one to one correspondence, called Revuz correspondence, between positive  smooth, with respect to the Dirichlet form $(B,V)$,  measures on $E$  and positive  continuous additive functionals of the Hunt processes associated with $(B,V)$ (see \cite{FOT,MR}). In the parabolic case the situation is more subtle. Let $\BX$ denote a Hunt process associated with a generalized Dirichlet form $\EE$ in the resolvent sense (see Section \ref{sec2}) which is generated by the operator $\frac{\partial}{\partial t}+L_t$.
In the paper we first show that with each smooth measure $\mu$ on $E_{0,T}$  with respect the form $\EE$ one can associate uniquely a positive  additive functional $A^\mu$ of $\BX$ which is natural, i.e. has no common discontinuities with the Hunt process associated with $\EE$. This is a counterpart to the known result concerning smooth measures on $\BR\times E$ (see \cite{O2,O4}). Then we analyse more carefully the nature of jumps of $A^\mu$ in the case where $\mu\in\MM_{0,b}(E_{0,T})$.  Roughly speaking, our main result says that for $\mu\in\MM_{0,b}(E_{0,T})$ the jumps of $A^\mu$ are related to $g_2$ from decomposition (\ref{eq1.4}). We show that $g_2$ has always a quasi-c\`adl\`ag modification $\tilde g_2$, i.e. an $m_1$-version   $\tilde g_2$ such that the process $t\mapsto \tilde g_2(X_t)$ is right-continuous with left limits, and for every predictable stopping time $\tau$,
\[
\Delta A^\mu_\tau:=A^\mu_{t}-A^\mu_{t-}=\Delta\tilde g_2(X)_\tau.
\]
In other words, $A^\mu$ has jumps that coincide with the jumps of the process $\tilde g_2(X)$ in predictable stopping times.
This implies  that $\tilde g_2$ is quasi-continuous if and only if $A^\mu$ is continuous.

\section{Preliminaries}
\label{sec2}

In this paper, $E$ is a locally compact separable metric space and
$m$ is an everywhere dense Radon measure on $E$, i.e. $m$ is a
positive Borel measure on $E$, which is  finite on compact sets
and strictly positive on nonempty open sets.

We set $E^1=\BR\times E$, $m_1=dt\otimes m$, and for $T>0$, we set
$E_{0,T}=(0,T)\times E$. We denote by $\BB(E^1)$ (resp. $\BB(E_{0,T})$) the set of all Borel measurable subsets of $E^1$ (resp. $E_{0,T}$). With the customary abuse of notation, the same symbols are used to denote the sets of real Borel measurable functions on $E^1$ (resp. $E_{0,T}$). $\BB_b(E^1)$ is  the set of all
real bounded Borel measurable functions on $E^1$ and $\BB_b^+(E^1)$ is the subset of $\BB_b(E^1)$ consisting of positive  functions. The sets $\BB_b(E_{0,T})$,
$\BB^+_b(E_{0,T})$ are defined analogously.

\subsection{Time dependent Dirichlet forms}
\label{sec2.1}

Let $H=L^2(E;m)$ and  $(\cdot,\cdot)_H$ denote the usual inner
product in $H$. In this paper, we assume that  we are given a family
$\{(B^{(t)},V),t\in[0,T]\}$ of regular (non-symmetric) Dirichlet
forms on $H$ with common domain $V\subset H$ (see \cite[Chapter
I]{MR} for the definitions) satisfying the following conditions.
\begin{enumerate}
\item[(a)]There is $K\ge0$ such that $ |B^{(t)}_1(\varphi,\psi)|\le K
(B^{(t)}_1(\varphi,\varphi))^{1/2}(B^{(t)}_1(\psi,\psi))^{1/2}$
for all $\varphi,\psi\in V,t\in[0,T]$, where as usual, we set
$B^{(t)}_{\lambda}(\varphi,\psi)=(B^{(t)}(\varphi,\psi)
+\lambda(\varphi,\psi)_H$ for $\lambda\ge0$.

\item[(b)]$[0,T]\ni t\mapsto B^{(t))}(\varphi,\psi)$ is measurable
for all $\varphi,\psi\in V$.

\item[(c)]There is $c\ge1$ such that
$c^{-1}B^{(0)}(\varphi,\varphi)\le B^{(t)}(\varphi,\varphi)\le c
B^{(0)}(\varphi,\varphi)$ for all $t\in[0,T]$ and $\varphi\in
V$.
\end{enumerate}
To shorten notation, we continue to write $B$ for $B^{(0)}$. By
putting $B^{(t)}=B$ for $t\notin[0,T]$, we may and will assume that
$B^{(t)}$ is defined and satisfies (c) for all $t\in\BR$. We denote by
$\tilde B^{(t)}$ the symmetric part of $B^{(t)}$,
i.e. $\tilde B^{(t)}(\varphi,\psi)=\frac12(B^{(t)}(\varphi,\psi)
+B^{(t)}(\psi,\varphi))$.

Since $V$ is a dense subspace of $H$ and $(B,V)$  is closed, $V$
is a real Hilbert space with respect to $\tilde B_1(\cdot,\cdot)$,
which is densely and continuously embedded in $H$. We equip $V$ with the norm
$\|\cdot\|_V$  defined by  $\|\varphi\|^2_V=B_1(\varphi,\varphi)$, $\varphi\in V$.
We denote by  $V'$ the dual space of $V$, and by $\|\cdot\|_{V'}$  the
corresponding norm. For $T>0$, we set
\[
\HH=L^2(0,T;H),\qquad \VV=L^2(0,T;V),\qquad \VV'=L^2(0,T;V')
\]
and
\[
\|u\|^2_{\VV}=\int^T_0\|u(t)\|^2_V\,dt,\qquad
\|u\|^2_{\VV'}=\int^T_0\|u(t)\|^2_{V'}\,dt.
\]
We shall identify  $H$ and its dual $H'$. Then $V\subset H\simeq
H'\subset V'$ continuously and densely, and hence
$\VV\subset\HH\simeq\HH'\subset\VV'$ continuously and densely.
%(osrodkowosc?)
%, and by (B2)(?),
%\[
%\|u\|_{\VV'}\|u\|_{\HH}\le \|u\|_{\VV} \quad ? (potrzebne?)
%\]

For given $u\in\VV$ let $\partial_tu$ denote the derivative in the
distribution  sense of the function $t\mapsto u(t)\in V$, and let
\[
\WW=\{u\in \VV:\partial_tu\in \VV'\}, \qquad
\|u\|_{\WW}=\|u\|_{\VV}+\|\partial_tu\|_{\VV'}.
\]
It is well known that there is a continuous embedding of $\WW$
into $C([0,T];H)$, i.e. for every $u\in\WW$ one can find $\bar
u\in C([0,T];H)$ such that $u(t)=\bar u(t)$ for a.e. $t\in[0,T]$
(with respect to the Lebesgue measure) and
\begin{equation}
\label{eq2.5} \|u\|_{C([0,T];H)}\le C\|u\|_{\WW}
\end{equation}
for some $C>0$. In what follows, we adopt the convention that any
element of $\WW$ is already in $C([0,T];H)$. With this convention
we may define the spaces
\[
\WW_0=\{u\in\WW:u(0)=0\},\qquad \WW_T=\{u\in\WW:u(T)=0\}.
\]
Note that $\WW_0,\WW_T$ are reflexive spaces as closed linear
subspaces of the reflexive space $\WW$.

The linear operator $\partial_t$ on $\HH$ with domain $\WW_T$ will be denoted by $\Lambda$. Its adjoint, i.e. the operator $-\partial_t$ with domain $\WW_0$, will be denoted by $\hat\Lambda$.

We denote by $\EE$  the generalized  Dirichlet  form
associated with $\Lambda$ and the family $\{(B^{(t)},V),t\in[0,T]\}$, that is
\begin{equation}
\label{eq2.23} \mathcal{E}(u,v)=\left\{
\begin{array}{l}\langle-\partial_tu,v\rangle+\BB(u,v),
\quad u\in\WW_T,v\in\VV,\smallskip \\
\langle\partial_tv, u\rangle+\BB(u,v),\quad u\in\VV,v\in\WW_0,
\end{array}
\right.
\end{equation}
where $\langle\cdot,\cdot\rangle$ is the duality pairing between
$\VV'$ and $\VV$, and
\begin{equation}
\label{eq2.24} \BB(u,v)=\int^T_0B^{(t)}(u(t),v(t))\,dt.
\end{equation}
The generalized form associated with $\hat\Lambda$ and the family $\{(\hat B^{(t)},V),t\in[0,T]\}$, where $\hat B^{(t)}(u,v)=B^{(t)}(v,u)$, $u,v\in V$, $t\in[0,T]$, will be denoted by $\hat\EE$.

In the paper, we denote by $(G_{\alpha})_{\alpha>0}$ (resp. $\hat G_{\alpha})_{\alpha>0}$
the resolvent (resp. coresolvent) associated with the
form $\EE$, i.e. $(G_{\alpha})_{\alpha>0}$, $(\hat
G_{\alpha})_{\alpha>0}$ are strongly continuous resolvents of
contractions on $\HH$ such that $G_{\alpha}(\HH)\subset\WW_T$,
$\hat G_{\alpha}(\HH)\subset\WW_0$ and
\begin{equation}
\label{eq2.26}
\EE_{\alpha}(G_{\alpha}f,g)=(f,g)_{\HH}=\EE_{\alpha}(g,\hat
G_{\alpha}f),\quad f\in\HH,g\in\VV
\end{equation}
(for a construction of the resolvents see, e.g., \cite[Chapter I]{St2}).

%\begin{lemma}
%\label{lem2.2} If $v,u\in\VV$, then
%\begin{equation}
%\label{eq3.5}  \alpha G_{\alpha}u\rightarrow u,\quad \alpha\hat G_{\alpha}v\rightarrow %v\quad\mbox{in }
%\VV.
%\end{equation}
%If moreover $u\in \WW_T$, $v\in\WW_0$, then
%\begin{equation}
%\label{eq2.3} \alpha G_{\alpha}u\rightarrow u\mbox{ weakly in
%} \WW_T,\quad  \alpha\hat G_{\alpha}v\rightarrow v\mbox{ weakly in
%} \WW_0.
%\end{equation}
%\end{lemma}
%\begin{proof}
%For (i) see, e.g., \cite[Proposition I.III.7]{St2}. Assume that
%$v\in\WW_0$. By (\ref{eq3.5}), for $w\in\VV$ we have
%\[
%\langle\hat\Lambda(\alpha\hat
%G_{\alpha}v),w\rangle=\BB(w,\alpha\hat G_{\alpha}v)
%-\EE(w,\alpha\hat G_{\alpha}v)\rightarrow\BB(w,v) -\EE(w,v)
%=\langle\hat\Lambda v,w\rangle.
%\]
%Since $\VV$ is reflexive, $\{\hat\Lambda(\alpha\hat
%G_{\alpha}v)\}$ is weakly convergent in $\VV'$ as
%$\alpha\rightarrow\infty$. From this, (\ref{eq3.5}) and Proposition
%\ref{prop5.1} we get (\ref{eq2.3}).
%\end{proof}

Let $\WW^1=\{u\in L^2(\BR;V):\partial_tu\in L^2(\BR;V')\}$. We denote
by $\EE^1$ the time dependent Dirichlet  form associated
with $\{(B^{(t)},V),t\in\BR\}$, that is
\[
\mathcal{E}^1(u,v)=\left\{
\begin{array}{l}\langle-\partial_tu,v\rangle+\BB^1(u,v),
\quad u\in\WW^1,v\in L^2(\BR;V),\smallskip \\
\langle\partial_tv, u\rangle+\BB^1(u,v),\quad u\in
L^2(\BR;V),v\in\WW^1,
\end{array}
\right.
\]
where now $\langle\cdot,\cdot\rangle$ stands for the duality pairing
between $L^2(\BR;V')$ and $L^2(\BR;V)$,  and
\[
\BB^1(u,v)=\int_{\BR}B^{(t)}(u(t),v(t))\,dt.
\]
The resolvent (resp. coresolvent) associated with the form $\EE^1$  (see \cite[Chapter I]{St2}) will be denoted by $(G^1_{\alpha})_{\alpha>0}$ (resp. $(\hat G^1_{\alpha})_{\alpha>0}$).

Note that $\EE^1$ can be identified with some generalized Dirichlet forms
in the sense considered in \cite{St2,T1,T2} (see \cite[Example I.4.9(iii)]{St2}).

Let $\psi\in L^1(E_{0,T};m_1)\cap\BB(E_{0,T})$ be a function such that $0<\psi\le1$. We define the capacity $\mbox{Cap}_\psi$ associated with $\EE$ as in \cite[Section III.2]{St2} (see also \cite{T2}), that is for an open set $U\subset E_{0,T}$ we set
\[
\mbox{Cap}_{\psi}(U)=((G_1\psi)_U,\psi)_{\HH},
\]
where $(G_1\psi)_U$ is the 1-reduced function of $G_1\psi$, and then for arbitrary $A\subset E_{0,T}$ we set $\mbox{Cap}_{\psi}(A)=\inf\{\mbox{Cap}_{\psi}(U),\,U\supset A,\,U\mbox{ open}\}$. By \cite[Lemma III.2.9]{St2} and (\ref{eq2.26}), for any open $U\subset E_{0,T}$, $((G_1\psi)_U,\psi)_{\HH}=\EE_1((G_1\psi)_U,\hat G_1\psi)=\EE_1(G_1{\psi},(\hat G_1\psi)_U)=(\psi,(\hat G_1\psi)_U)$, so  $\mbox{Cap}_{\psi}$ associated with $\EE$ coincides with the capacity defined as $\mbox{Cap}_{\psi}$ but for the dual  form $\hat\EE$. In particular, $\mbox{Cap}_{\psi}$ coincides with the capacity considered in \cite{T1}.

By \cite[Proposition III.2.8]{St2}, $\mbox{Cap}_\psi$ is a Choquet capacity. Also note that by \cite[Proposition III.2.10]{St2}, the family of sets $A\subset E_{0,T}$ such that $\mbox{Cap}_\psi(A)=0$ is the same for all $\psi$ as above.

We will denote by $\mbox{Cap}^1$ the capacity associated with $\EE^1$ and  defined in \cite{O3} (the definition is also given in \cite{O2} and \cite[Section 6.2]{O4}). By \cite[Lemma 4.2]{O3} (or \cite[Lemma 6.2.8]{O4}) and \cite[Theorem A.1.2]{FOT}, $\mbox{Cap}^1$
is a Choquet capacity.
%\\
%K:\\
%Trzeba pokazac ze \cite[Lemma 4.2]{O3} dla zbiorow otwartych. Pokazemy wlasnosc (iii). Wiemy, ze zachodzi dla zbiorow  %otwartych relatywnie zwartych. Niech $A$ bedzie otwarty. Z definicji %$\mbox{Cap}^1(A)$ istnieje ciag zbiorow otwartych %relatywnie zwartych $\{B_k\}$ taki, ze %$B_k\subset A$ i $\mbox{Cap}^1(B_k)\rightarrow\mbox{Cap}^1(A)$. Mozna zalozyc, %ze $B_k\uparrow$. Mamy $\mbox{Cap}^1(A_n\cap B_k)\le\mbox{Cap}^1(A_n)$  i z wlasnosci (iii), %$\mbox{Cap}^1(A_n\cap %B_k)\uparrow\mbox{Cap}^1(A\cap B_k)=\mbox{Cap}^1(B_k)$. Dlatego %$\mbox{Cap}^1(B_k)\le\sup_n\mbox{Cap}^1(A_n)$ i w %konsekwencji %$\mbox{Cap}^1(A)\le\sup_n\mbox{Cap}^1(A_n)$. Przeciwna nierownosc jest oczywista, bo %$A_n\subset A$.\\
%KK\\

Some relations between $\mbox{Cap}_\psi$ and $\mbox{Cap}^1$, as well as relations between
$\mbox{Cap}_\psi$ and some other notions of parabolic capacity considered in the literature will be studied in Section \ref{sec3}.

We say that a set $A\subset E_{0,T}$ (resp. $A\subset E^1$) is $\EE$-exceptional (resp. $\EE^1$-exceptional) if
$\mbox{Cap}_\psi(A)=0$ (resp. $\mbox{Cap}^1(A)=0$), and we say that a property holds
$\EE$-quasi-everywhere (resp. $\EE^1$-quasi-everywhere) if the set of those $x\in E_{0,T}$ (resp. $x\in E^1$)
for which it does not hold is $\EE$-exceptional (resp. $\EE^1$-exceptional).

Recall that an increasing sequence $\{F_n\}$ of closed subsets of $E_{0,T}$ (resp. $E^1$) is
called an $\EE$-nest (resp. $\EE^1$-nest) if $\mbox{Cap}_\psi(F_n^c)\rightarrow0$ (resp. $\mbox{Cap}^1(F_n^c)\rightarrow0$) as $n\rightarrow\infty$.
An increasing sequence $\{F_n\}$ of closed subsets of $E^1$ is called a generalized $\EE^1$-nest if Cap$^1(K\setminus F_n)\rightarrow 0$
for every compact $K\subset E_{0,T}$.

A function $u:E_{0,T}\rightarrow\BR$ (resp. $u:E^1\rightarrow\BR$) is called
$\EE$-quasi-continuous (resp. $\EE^1$-quasi-continuous) if there exists an $\EE$-nest (resp. $\EE$-nest) $\{F_n\}$ such that
$u|_{F_n}$ is continuous for every $n\in\BN$.

It is known (see \cite[Proposition IV.1.8]{St2})  that each $u\in\WW_0$ has an $\EE$-quasi-continuous $m_1$-version. Similarly, each  $u\in\WW_T$ has an $\EE$-quasi-continuous $m_1$-version. We will denote them by $\tilde u$.

\subsection{Dirichlet forms and Markov processes}

Let $\Delta$ be  adjoint to $E^1$ as the point at infinity.
We adopt the convention that every function $f$ on
$E^1$ (resp. $E_{0,T}$) is extended to $E^1\cup\{\Delta\}$ (resp. $E_{0,T}\cup\{\Delta\}$) by setting $f(\Delta)=0$.

Let $\Omega_1=\{\omega:[0,\infty)\rightarrow E^1\cup \{\Delta\}:\omega(s)=\Delta,\, s\ge t\mbox{ if }\omega(t)=\Delta\}$ and
\[
X^1:\Omega_1\rightarrow \Omega_1,\quad X^1_t(\omega)=\omega(t).
\]
By \cite[Theorem 4.2]{O1}  (see also  \cite[Theorem 5.1]{O3} or \cite[Theorem 6.3.1]{O4}), there exists a Hunt process
$\BX^1=(\Omega_1,(\FF^1_t)_{t\ge0}, (X^1_t)_{t\ge0},(P^1_x)_{x\in
E^1\cup\{\Delta\}})$ with life time $\zeta^1$ and cemetery state
$\Delta$  associated with  $\EE^1$ in the resolvent sense, i.e.  for all
$\alpha>0$ and $f\in\BB_b(E^1)\cap L^2(E^1;m_1)$ the resolvent of $\BX^1$
defined as
\begin{equation}
\label{eqr.1}
R^1_{\alpha}f(x)=E^1_x\int^{\infty}_0e^{-\alpha t}f(X^1_t)\,dt,\quad
x\in E^1,\quad f\in\BB_b(E^1)
\end{equation}
($E^1_x$ stands for the expectation with respect to $P^1_x$) is an $\EE^1$-quasi-continuous $m_1$-version of the resolvent
$G^1_{\alpha}f$. Let
\[
\upsilon:\Omega_1\rightarrow \Omega_1,\quad (\upsilon(\omega))(t)=\pi(\omega(0))+t,
\]
where $\pi:E^1\rightarrow \mathbb R$ is the projection on $\mathbb R$. Note that
under the measure  $P_x$ with $x=(s,x^0)\in E^1$,
$\upsilon$ is the uniform motion  to the right, i.e.
$\upsilon(t)=\upsilon(0)+t$, $\upsilon(0)=s$.
By \cite[Theorem 4.2]{O1} (or \cite[Theorem 5.1]{O3}),
\begin{equation}
\label{eq2.8}
X^1_t=(\upsilon(t),X^0_{\upsilon(t)}),\quad t\ge0,
\end{equation}
under the measure  $P_x$ for every $x\in E^1$.

Recall that $B\in\BB(E^1)$ is called $\BX^1$-exceptional if $P^1_{x}(\sigma_B<\infty)=0$ for $m_1$-a.e. $x\in E^1$, where
\[
\sigma _B=\inf\{t>0: X^1_t\in B\}.
\]
By the remarks in \cite[page 298]{O3} (see also \cite[Lemma 2.3]{O2}),
\begin{equation}
\label{eq2.12}
B\mbox{ is $\EE^1$-exceptional if and only if }B\mbox{ is $\BX^1$-exceptional.}
\end{equation}

Similarly, by \cite[Section IV.2]{St2}, there exists a Hunt process
$\BX=(\Omega$, $(\FF_t)_{t\ge0}$, $(X_t)_{t\ge0}$, $(P_x)_{x\in
E_{0,T}\cup\{\Delta\}})$ with life time $\zeta$ associated with  $\EE$ in the resolvent sense, i.e.  for all
$\alpha>0$ and $f\in\BB_b(E_{0,T})\cap L^2(E_{0,T};m_1)$,
\begin{equation}
\label{eqr.2}
R_{\alpha}f(x)=E_x\int^{\infty}_0e^{-\alpha t}f(X_t)\,dt,\quad
x\in E_{0,T},\quad f\in\BB_b(E_{0,T})
\end{equation}
($E_x$ stands for the expectation with respect to $P_x$) is an $\EE$-quasi-continuous $m_1$-version of the resolvent $G_{\alpha}f$.  In this case
$\Omega=\{\omega:[0,\infty)\rightarrow E_{0,T}\cup \{\Delta\}:\omega(t)=\Delta,\,s\ge t$ if
$\omega(s)=\Delta\}$ and
\[
X:\Omega\rightarrow \Omega,\quad X_t(\omega)=\omega(t).
\]
It is clear that $\Omega\subset \Omega_1$ and $X=X^1$ on $\Omega$.

A set $B\in\BB(E_{0,T})$ is called $\BX$-exceptionl if $P_{x}(\sigma_B<\infty)=0$ for $m_1$-a.e. $x\in E_{0,T}$, where $\sigma_B$ is defined as before but with $\BX^1$ replaced by $\BX$. By \cite[Theorem IV.3.8]{St2},
\begin{equation}
\label{eq2.13}
B\mbox{ is $\EE$-exceptional if and only if }B\mbox{ is $\BX$-exceptional.}
\end{equation}

Let $f\in \BB(E_{0,T})\cap L^2(E_{0,T};m_1)$, $g\in \BB(E^1)\cap L^2(E^1;m_1)$ and $v=R_\alpha f$, $u=R^1_\alpha g$. Since $v$ is an $m_1$ version of $G_{\alpha}f$ and $u$ is an $m_1$-version of $G^1_{\alpha}g$, we have
$v\in \WW_T$, $u\in \WW^1$ and
\begin{equation}
\label{eq2.eqs}
\EE_\alpha(v,\eta)=(f,\eta)_{L^2(E_{0,T};m_1)},\,\, \eta\in \WW_0,\quad\quad
\EE^1_\alpha(u,\eta)=(g,\eta)_{L^2(E^1;m_1)},\,\, \eta\in \WW^1.
\end{equation}
Let $g=f$ on $E_{0,T}$ and $g=0$ on $E^1\setminus E_{0,T}$. Then, by (\ref{eqr.1}) and (\ref{eq2.8}),
$u_{|E_{0,T}}\in \WW_T$. From this and (\ref{eq2.eqs}) we deduce that $v=u_{|E_{0,T}}$. By this and  (\ref{eqr.1}) and (\ref{eqr.2}), we get
\begin{equation}
\label{eq8.c17}
E^1_x\int_0^\infty e^{-\alpha t} f(X^1_t)\,dt=E_x\int_0^\infty e^{-\alpha t} f(X_t)\,dt
\end{equation}
for every $\alpha\ge 0$ and every $f\in\BB^+(E^1)$ such that $\mathbf{1}_{E_{0,T}}f=f$.
Let $\Pi:\Omega_1\rightarrow \Omega$ be defined as
\[
\Pi(\omega)(t)=\omega(t),\quad t<T-\upsilon(0),\qquad \Pi(\omega)(t)=\Delta,\quad t\ge  T-\upsilon(0).
\]
Observe that
\[
E^1_x\int_0^\infty e^{-\alpha t} f(X^1_t)\,dt=\int_\Omega\int_0^\infty e^{-\alpha t} f(X_t)\,dt\,d(P^1_x\circ \Pi^{-1}).
\]
Hence
\begin{equation}
\label{eq2.cpr}
P_x=P^1_x\circ\Pi^{-1},\quad x\in E_{0,T}.
\end{equation}
Therefore, for every $x\in E_{0,T}$\,,
\begin{align*}
&P_x(\zeta\le T-\upsilon(0))=P_x\Big(\omega\in\Omega:\inf\{t\ge 0;\, X_t(\omega)\notin E_{0,T}\}\le T-\upsilon(0)(\omega)\Big)\\
&\quad=P_x\Big(\omega\in\Omega;\, \inf\{t\ge 0:X^1_t(\omega)\notin E_{0,T}\}\le T-\upsilon(0)(\omega)\Big)\\&
\quad=(P^1_x\circ \Pi^{-1})\Big(\omega\in\Omega:\inf\{t\ge 0;\, X^1_t(\omega)\notin E_{0,T}\}\le T-\upsilon(0)(\omega)\Big)\\
&\quad=P^1_x\Big(\omega\in\Omega_1:\inf\{t\ge 0;\, X^1_t(\Pi(\omega))\notin E_{0,T}\}\le T-\upsilon(0)(\Pi(\omega))\Big)\\
&\quad=P^1_x\Big(\omega\in\Omega_1:\inf\{t\ge 0: (\upsilon(t)(\Pi(\omega)),X^0_{\upsilon(t)(\Pi(\omega))}
(\Pi(\omega)))\notin E_{0,T}\}\\
&\qquad\qquad\qquad\qquad\qquad\qquad\qquad\qquad\qquad\qquad\qquad\le T-\upsilon(0)(\Pi(\omega))\Big)\\
&\quad\ge P^1_x\Big(\omega\in\Omega_1:\inf\{t\ge 0:\upsilon(t)(\Pi(\omega)) \notin (0,T)\}\le T-\upsilon(0)(\Pi(\omega))\Big)=1.
\end{align*}
Thus
\begin{equation}
\label{eq8.c14}
\zeta=\zeta\wedge(T-\upsilon(0))\quad P_x\mbox{-a.s.}
\end{equation}

In this paper, we denote by $\hat\BX^1=(\Omega_1,(\hat \FF^1_t)_{t\ge0}, ( X^1_t)_{t\ge0},(\hat P^1_x)_{x\in E^1\cup\{\Delta\}},\zeta^1)$ the dual process of $\BX^1$,
i.e. a Hunt process whose resolvents $\hat R^1_{\alpha}f$ are $\EE^1$-quasi-continuous $m_1$-versions of $\hat G^1_{\alpha}f$ for any $f\in\BB_b(E^1)\cap L^2(E^1;m_1)$.
By \cite[Theorem 5.1]{O3},
\begin{equation}
\label{eq2.11}
X^1_t=(\hat\upsilon(t),\hat X^0_{\hat\upsilon(t)}),\quad t\ge0,
\end{equation}
where $\hat\upsilon$ is the uniform motion  to the left, i.e.
$\hat\upsilon(t)=\hat\upsilon(0)-t$, $\hat\upsilon(0)=s$ under the measure  $\hat P_x$ with $x=(s,x^0)\in E^1$.

We denote by $\hat\BX=( \Omega,(\hat \FF_t)_{t\ge0}, (X_t)_{t\ge0},(\hat P_x)_{x\in E_{0,T}\cup\{\Delta\}},\zeta)$  the dual process of $\BX$. It is a Hunt process associated with the dual form $\hat\EE$, For all $\alpha>0$ and $f\in\BB_b(E_{0,T})\cap L^2(E_{0,T};m_1)$ the resolvent of $\hat\BX$ defined as
\[
\hat R_{\alpha}f(x)=\hat E_x\int^{\infty}_0e^{-\alpha t}f(X_t)\,dt,
\quad x\in E_{0,T},
\]
where  $\hat E_x$ denotes the expectation with respect to $\hat P_x$, is an $\EE$-quasi-continuous $m_1$-version of the coresolvent $\hat G_{\alpha}f$.
A modification of the argument used to prove (\ref{eq2.cpr}) and (\ref{eq8.c14}) shows that
\begin{equation}
\label{eq2.30}
\hat P_x=\hat P^1_x\circ\hat\Pi^{-1},\qquad \zeta=\zeta\wedge\hat\upsilon(0),\quad \hat P_x\mbox{-a.s.}
\end{equation}
for $x\in E_{0,T}$, where $\hat\Pi:\Omega_1\rightarrow\Omega$ is defined as
\[
\hat\Pi(\omega)(t)=\omega(t),\quad t<\hat\upsilon(0),\qquad \hat\Pi(\omega)(t)=\Delta,\quad t\ge\hat\upsilon(0).
\]

The relation (\ref{eq8.c14}) (resp. (\ref{eq2.30})) implies that the operator $G_\alpha$ (resp. $\hat G_\alpha$) is well defined
for $\alpha=0$,  and that $R_0f$ (resp. $\hat Rf$) is  $\EE$-quasi-continuous $m_1$-version of $G_0 f$ (resp. $\hat G_0 f$) for every $f\in\HH$, Indeed,  by (\ref{eq8.c14}), for every positive $f\in\BB(E_{0,T})$,
\[
Rf\le e^{\alpha T}R_\alpha f
\]
(we write $R$ instead of $R_0$). Hence,  for every $f\in \HH\cap \BB(E_{0,T})$, $Rf\in \HH$. By the resolvent identity,
\[
Rf=R_\alpha f+\alpha R_\alpha R f.
\]
Therefore $Rf\in G_\alpha(\HH)$, and by (\ref{eq2.26}),
\begin{equation}
\label{eq2.26cc}
\EE(Rf,g)=(f,g)_{\HH},\quad f\in\HH,g\in\VV.
\end{equation}
In other words, $G_0f=Rf$ $m_1$-a.e. Of course, $Rf$ is $\EE$-quasi-continuous since $Rf=R_\alpha(f+\alpha Rf)$.
A similar argument applies to  $\hat R_0$ and $\hat G_0$.

\subsection{Smooth measures}

A Borel (signed) measure $\mu$ on $E_{0,T}$ is called $\EE$-smooth if it does not charge
exceptional sets, i.e. for any Borel set $B\subset E_{0,T}$, if $\mbox{Cap}_{\psi}(B)=0$, then $\mu(B)=0$,  and there exists an $\EE$-nest $\{F_n\}$ of compact subsets of $E_{0,T}$
such that $|\mu|(F_n)<\infty$ for $n\in\BN$, where $|\mu|$ denotes
the variation of $\mu$. A Borel measure $\mu$ on $E^1$ is called $\EE^1$-smooth if it does not charge $\EE^1$-exceptional sets and there exists a generalized $\EE^1$-nest $\{F_n\}$
such that  $|\mu|(F_n)<\infty$ for $n\in\BN$.

The set of all $\EE$-smooth (resp. $\EE^1$-smooth) measures will be denoted by $S(E_{0,T})$ (resp. $S^1(E^1)$).  $\MM_b(E_{0,T})$
(resp. $\MM^1_b(E^1)$)
denotes the set of all bounded Borel measures on $E_{0,T}$ (resp. $E^1$), i.e. Borel
measures $\mu$ such that $|\mu|(E_{0,T})<\infty$ (resp. $|\mu|(E^1)<\infty$). $\MM_{0,b}(E_{0,T})$ (resp. $\MM^1_{0,b}(E^1)$) denotes
the subset of $\MM_b(E_{0,T})$ (resp. $\MM^1_{b}(E^1)$) consisting of all $\EE$-smooth (resp. $\EE^1$-smooth) measures.

Let $\mu$ be a positive Borel measure on $E_{0,T}$ such that $\mu$ does not charge
sets of zero capacity. We call it a measure of finite energy
integral if there is $C\ge0$ such that
\begin{equation}
\label{eq2.28}
\Big|\int_{E_{0,T}}\tilde\eta\,d\mu\Big|\le C\|\eta\|_{\WW},\quad \eta\in \WW_0.
\end{equation}
If there is $C\ge0$ such that
\[
\Big|\int_{E_T}\tilde\eta\,d\mu\Big|\le C\|\eta\|_{\WW},\quad\eta\in\WW_T,
\]
we call it a measure of finite co-energy integral. In both cases, $\tilde\eta$ denotes an $\EE$-quasi-continuous $m_1$-version of $\eta$.
The set of all positive smooth measures on $E_{0,T}$ of finite energy (resp. co-energy) integral will be denoted by $S_0(E_{0,T})$ (resp. $\hat S_0(E_{0,T})$).

\begin{lemma}
\label{lem2.3} Let $\mu$ be a  positive Borel  measure on $E_{0,T}$. Then
\begin{enumerate}
\item[\rm(i)]$\mu\in S_0(E_{0,T})$ if and only if there exists $u\in\VV$ such that
\begin{equation}
\label{eq2.1} \int_{E_{0,T}}\tilde v\,d\mu=\EE_1(u,v),\quad v\in\WW_0.
\end{equation}
\item[\rm(ii)]$\mu\in\hat S_0(E_{0,T})$ if and only if there exists $u\in\VV$ such that
\begin{equation}
\label{eq2.25} \int_{E_{0,T}}\tilde v\,d\mu=\EE_1(v,u),\quad v\in\WW_T.
\end{equation}
\end{enumerate}
\end{lemma}
\begin{proof}
We provide the proof of (i). The proof of (ii) is analogous. If (\ref{eq2.1}) is satisfied, then of course $\mu\in S_0(E_{0,T})$. Suppose that
$\mu$ is of finite energy integral. Then, by \cite[Lemma 4.2]{T2} applied to the dual form $\hat\EE$, for every positive $\eta\in\WW_0$ there exists $u\in $ such that
\begin{equation}
\label{eq2.29}
\int_{E_{0,T}}\alpha \widetilde{\hat G_\alpha\eta}\,d\mu=\EE_1(u,\alpha \hat G_\alpha\eta ),
\end{equation}
where $\widetilde{\hat G_{\alpha}\eta}$ is  a quasi-continuous modification of $\hat G_{\alpha}\eta$. By \cite[Proposition 2.7]{St1}, $\alpha\hat G_\alpha\eta\rightarrow \eta$ strongly in $\WW_0$. Therefore letting $\alpha\rightarrow\infty$ in (\ref{eq2.29}) and applying  \cite[Corollary III.3.8]{St2}, Fatou's lemma  and \cite[Lemma 2]{Pierre3}, we get (\ref{eq2.28}).
\end{proof}

Let $\mu\in S_0(E_{0,T})$. The element $u\in\VV$ defined by  (\ref{eq2.1}) is uniquely determined. We will denote it by $U_1\mu$. Similarly, for $\mu\in\hat S_0(E_{0,T})$, the element $u\in\VV$ defined by  (\ref{eq2.25}) is uniquely determined. We will denote it by $\hat U_1\mu$. We also set
\[
U_\alpha\mu= U_1\mu+(1-\alpha)G_{\alpha} U_1\mu,
\qquad\hat U_\alpha\mu= \hat U_1\mu+(1-\alpha)\hat G_{\alpha}\hat U_1\mu
\quad \alpha\ge0.
\]
It is clear that for any $\alpha\ge0$,
\begin{equation}
\label{eq2.22}
\int_{E_{0,T}}\tilde v\,d\mu=\EE_\alpha(U_\alpha\mu,v),\,\, v\in\WW_0,\qquad\int_{E_{0,T}}\tilde v\,d\mu=\EE_\alpha(v,\hat U_\alpha\mu),\,\,v\in\WW_T.
\end{equation}

\section{Parabolic capacity}
\label{sec3}

Our basic capacity associated with $\EE$ is $\mbox{Cap}_{\psi}$. Exceptional sets with respect
to $\mbox{Cap}_{\psi}$ have nice probabilistic interpretation given by (\ref{eq2.13}).
However, in the literature devoted to partial differential equations, usually  some other notions of capacity are used.
In this section, we recall some of them and  prove that they are all equivalent to $\mbox{Cap}_{\psi}$. These results will be needed in the next sections.

\begin{lemma}
\label{lm.eqc12}
For any $A\subset E_{0,T}$, $\mbox{\rm Cap}_\psi(A)=0$ if and only if $\mbox{\rm Cap}^1_{|E_{0,T}}(A)=0$.
\end{lemma}
\begin{proof}
Since both  $\mbox{\rm Cap}_\psi$ and $\mbox{\rm Cap}^1_{|E_{0,T}}$  are Choquet capacities, it is enough to prove that for any $B\in\BB(E_{0,T})$,
$\mbox{Cap}_\psi(B)=0$ if and only if $\mbox{Cap}^1(B)=0$. Suppose that  Cap$^1(B)=0$.
Then, by (\ref{eq2.12}),  $P^1_{x}(\sigma_B<\infty)=0$ for $m_1$-a.e. $x\in E^1$.
Observe that
\[
\{\omega\in\Omega_1;\, \exists _{t>0}\,\, X^1_t(\omega)\in B\}=\Pi^{-1}(\{\omega\in\Omega;\, \exists _{t>0}\,\, X_t(\omega)\in B\}).
\]
Hence, by (\ref{eq2.cpr}), $P_{x}(\sigma_B<\infty)=0$ for $m_1$-a.e. $x\in E_{0,T}$.
Consequently, $\mbox{Cap}_\psi(B)=0$  by (\ref{eq2.13}). Now suppose  that Cap$_\psi(B)=0$.  Then, by \cite[Lemma III.2.9]{St2} and (\ref{eq2.13}),
$P_{x}(\sigma_B<\infty)=\hat P_{x}(\sigma_B<\infty)=0$ for $m_1$-a.e.
$x\in E_{0,T}$. From this and (\ref{eq2.8}) and (\ref{eq2.cpr}) it follows that
\begin{equation}
\label{eq2.he1}
P^1_{x}(\sigma_B<\infty)=0\quad \mbox{for }m_1\mbox{-a.e. } x\in[0,\infty)\times E.
\end{equation}
Similarly, from (\ref{eq2.11}) and (\ref{eq2.30}) it follows that
\begin{equation}
\label{eq2.he2}
\quad \hat P^1_{x}(\sigma_B<\infty)=0\quad\mbox{for } m_1\mbox{-a.e. }x\in(-\infty,T]\times E.
\end{equation}
By \cite[Lemma III.2.9]{St2}, for every $\eta\in \BB^+_b(E^1)$,
\begin{equation}
\label{eq2.672}
E^1_{\eta\cdot m_1}e^{-\sigma_B}\int_{\sigma_B}^{\infty}e^{-t}\eta(X^1_t)\,dt
=\hat E^1_{\eta\cdot m_1}e^{- \sigma_B}\int_{\sigma_B}^{\infty}e^{-t}\eta(X^1_t)\,dt,
\end{equation}
where $E^1_{\eta\cdot m_1}$ denotes the expectation with respect to the measure $P^1_{\eta\cdot m_1}$ defined as $P_{\eta\cdot m_1}(\cdot)=\int_{E^1}P_x(\cdot)\eta(x)\,m_1(dx)$ and $\hat E^1_{\eta\cdot m_1}$
denotes the expectation with respect to $\hat P^1_{\eta\cdot m_1}$ defined as $ P^1_{\eta\cdot m_1}$ but with $P_x$ replaced by $\hat P_x$.
Let $\eta=0$ on $[T,\infty)\times E$ and $\eta>0$ on $(-\infty,T)\times E.$ Then, by (\ref{eq2.he2}), the right-hand side of (\ref{eq2.672}) equals zero, which implies that $P^1_{x}(\sigma_B<\infty)=0$ for $m_1$-a.e. $x\in(-\infty,T]\times E$. This when combined with (\ref{eq2.he1}) shows that $P^1_{x}(\sigma_B<\infty)=0$ for $m_1$-a.e. $x\in E^1$. By (\ref{eq2.12}),  $\mbox{Cap}^1(B)=0$.
\end{proof}

We now recall the capacity considered in Pierre  \cite{Pierre3} (see also \cite{Pierre1,Pierre2}).
Let $U$ be a relatively compact open  subset of $E_{0,T}$. By \cite[Proposition III.1.6]{St2}, for every $n\ge1$
there exists a unique solution $e^n_U\in \WW_T$ of the following problem
\[
\EE_1(e_U^n,\eta)=n((e_U^n-\mathbf{1}_{U})^-,\eta),\quad \eta\in\VV.
\]
By \cite[Proposition III.1.7]{St2},  $\{e^n_U\}$  converges strongly in $\HH$
and weakly in $\VV$. Set $e_U=\lim_{n\rightarrow \infty} e^n_U$.
By \cite[Proposition III.1.7]{St2}, for every $\eta\in \WW_0$ such that $\eta\ge \mathbf{1}_U$ $m_1$-a.e. we have
\begin{equation}
\label{eq.pt1}
\EE_1(e_U,\eta)\ge\BB_1(e_U,e_U).
\end{equation}
By Lemma \ref{lem2.3} (see also \cite{Pierre3}), there exists a measure $\mu_U\in S_0(E_{0,T})$ such that $e_U=U_1\mu_U$.
By \cite[Proposition II.1]{Pierre2}, supp$[\mu]\subset \overline U$.
In  \cite{Pierre3}, the capacity of $U$ is defined by
\[
c_0(U)=\mu_U(E_{0,T}).
\]
As usual, for an arbitrary $A\subset E_{0,T}$, we define
\[
c_0(A)=\inf\{c_0(U),\, A\subset U, U\mbox{ open}\}.
\]
By \cite[Proposition 2]{Pierre3}, $c_0$ is a Choquet capacity.

\begin{lemma}
\label{equiv.1}
For any $A\subset E_{0,T}$, $\mbox{\rm Cap}_\psi(A)=0$ if and only if $c_0(A)=0$.
\end{lemma}
\begin{proof}
Let $U$ be a relatively compact open subset of $E_{0,T}$. Let $\widetilde{\hat G_1\psi}$ be a quasi-continuous modification of $\hat G_1\psi$.
By the definitions of Cap$_\psi$ and  $U_1\mu_U$, and the fact that $\hat G_1$ is Markovian, we have
\[
\mbox{Cap}_\psi(U)\le\int_{E_{0,T}}e_U\psi\,dm_1=\EE_1(e_U,\hat G_1\psi)=\int_{E_{0,T}}\widetilde{\hat G_1\psi}\,d\mu_U\le\mu_U(E_{0,T})=c_0(U).
\]
Hence Cap$_\psi\le c_0$. Let $K$ be a compact subset of $E_{0,T}$ such that $\mbox{Cap}_\psi(K)=0$.
Then there exists a nonincreasing sequence $U_n$ of open relatively compact subsets of $E_{0,T}$ such that Cap$_\psi(U_n)\searrow 0$ and $K\subset U_n$. Let $\eta\in \WW_0\cap C_c(E_{0,T})$ be such that
$\eta\ge \mathbf{1}_{U_1}$ (see \cite[Lemma II.3]{Pierre2}). Then, by (\ref{eq.pt1}),
\begin{equation}
\label{eq2.1234}
\BB_1(e_{U_n},e_{U_n})\le\EE_1(e_{U_n},\eta)=\int_{E_{0,T}}\tilde\eta\,d\mu_{U_n}=c_0(U_n).
\end{equation}
Since $\{U_n\}$ is a nonincreasing, $\{e_{U_n}\}$ is nonincreasing. By \cite[Proposition IV.3.4]{St2}, $e_{U_n}\searrow 0$ $m_1$-a.e. By (\ref{eq2.1234}),
\[
\BB_1(e_{U_n},e_{U_n})\le c_0(U_1).
\]
Therefore there exists a subsequence (still denoted by $n$) such that $\{e_{U_n}\}$ is weakly convergent in $\VV$.
Since $e_{U_n}\searrow 0$ $m_1$-a.e., in fact $e_{U_n}\rightarrow 0$ weakly in $\VV$. Thus $\EE_1(e_{U_n},\eta)\rightarrow 0$, $\eta\in\WW_0$. By (\ref{eq2.1234}),
\[
\EE_1(e_{U_n},\eta)\ge c_0(K).
\]
Hence $c_0(K)$=0. Since Cap$_\psi$ and  $c_0$ are Choquet capacities, this implies the desired result.
\end{proof}
Write
\[
\WW^2= \{u\in\WW:\partial_tu\in L^2(E_{0,T};m_1)\},\qquad {\mathcal C}=C_c(E_{0,T})\cap\WW^2,
\]
where $C_c(E_{0,T})$ is the set of all real continuous functions on $E_{0,T}$ having  compact support. For a  compact set $K\subset E_{0,T}$, we set
\[
\mbox{CAP}(K)=\inf\{\|u\|_\WW: u\in \mathcal C,\, u\ge \mathbf{1}_{K}\}.
\]
Next, for an open $U\subset E_{0,T}$, we set
$\mbox{CAP}(U)=\sup\{\mbox{CAP}(K):K\subset U,\, K\mbox{-compact}\}$, and finally, for arbitrary $A\subset E_{0,T}$, we set
$\mbox{CAP}(A)=\inf\{\mbox{CAP}(U):A\subset U,\, U\mbox{ open}\}$.
We will also consider the capacity $\mbox{c}_2$ (see \cite{Pierre3}) defined as
\[
\mbox{c}_2(U)=\inf\{\|u\|_\WW:u\in\WW,\, u\ge\mathbf{1}_{U}\, m_1\mbox{-a.e.}\}
\]
for an open set $U\subset E_{0,T}$, and then by
$\mbox{c}_2(A)=\inf\{\mbox{CAP}_1(U):A\subset U,\, U\mbox{  open}\}$ for an arbitrary set $A\subset E_{0,T}$.  By \cite[Proposition 2]{Pierre3}, $\mbox{c}_2$ is a Choquet capacity.

\begin{lemma}
\label{lem3.4}
For any $A\subset E_{0,T}$, $\mbox{\rm c}_2(A)=0$ if and only if $\mbox{\rm Cap}_\psi(A)=0$.
\end{lemma}
\begin{proof}
Follows from Lemma \ref{equiv.1} because by \cite[Theorem 1]{Pierre3}, for any $A\subset E_{0,T}$,  $\mbox{c}_2(A)=0$ if and only if $c_0(A)=0$.
\end{proof}

\begin{lemma}
\label{lm.eqv}
For any $A\subset E_{0,T}$, $\mbox{\rm CAP}(A)=0$ if and only if $\mbox{\rm Cap}_\psi(A)=0$.
\end{lemma}
\begin{proof}
By virtue of Lemma \ref{lem3.4}, it suffices to prove that
$\mbox{CAP}$ is equivalent to $\mbox{c}_2$. Let $K$ be a compact subset of $E_{0,T}$ and let $\varepsilon>0$, $a>1$.
Choose $\eta_\varepsilon\in\mathcal C$  such that $\|\eta_\varepsilon\|_\WW\le \mbox{CAP}(K)+\varepsilon$ and $\eta_\varepsilon\ge\mathbf{1}_{K}$. Since $\eta_\varepsilon$ is continuous, there exists an open set $V_a$ such that $a\eta_\varepsilon\ge \mathbf{1}_{V_a}$ and $K\subset V_a$. We have
\[
\mbox{c}_2(K)\le\mbox{c}_2(V_a)\le \|a\eta_\varepsilon\|_\WW\le a(\mbox{CAP}(K)+\varepsilon).
\]
Letting $\varepsilon\downarrow 0$ and then $a\downarrow 1$ we see  that $\mbox{c}_2(K)\le\mbox{CAP}(K)$
for every compact $K\subset E_{0,T}$. From this and the fact that $\mbox{c}_2$ is a Choquet capacity   we conclude that
$\mbox{c}_2\le\mbox{CAP}$.

Now suppose that $\mbox{c}_2(K)=0$ for some compact set $K\subset E_{0,T}$. Let $U_n=\{(s,x)\in E_{0,T}:\mbox{dist}((s,x),K)<n^{-1}\}$. Since $E_{0,T}$ is locally compact and $K$ is compact, we may assume that  $U_n$ are relatively compact. Since $\mbox{c}_2$ is a Choquet capacity, $b_n:=\mbox{c}_2(U_n)\searrow 0$. Fix $\varepsilon>0$ and choose  $\eta^\varepsilon_n\in\WW$ such that $\eta^\varepsilon_n\ge \mathbf{1}_{U_n}$ $m_1$-a.e. and $\|\eta^\varepsilon_n\|_\WW\le b_n+\varepsilon$. By putting $\eta^\varepsilon_n(t)=\eta^\varepsilon_n(T)$ for $t\ge T$ and
$\eta^\varepsilon_n(t)=\eta^\varepsilon_n(0)$ for $t\le 0$, we may assume that
$\eta^\varepsilon_n\in \WW^1$. Write
\[
J_n(\eta^\varepsilon_n)=2n\int_0^{1/(2n)}\eta^\varepsilon_n(t+s)j(2ns)\,ds,
\]
where $j$ is a smooth positive function with support in $[0,1]$ such that $\int_\BR j(t)\,dt=1$. It is clear that
$J_n(\eta^\varepsilon_n)\in \WW^2$, $\|J_n(\eta^\varepsilon_n)\|_\WW\le \|\eta^\varepsilon_n\|_\WW$ and
\begin{equation}
\label{eq.cl1}
J_n(\eta^\varepsilon_n)\ge \mathbf{1}_{U_{n-1}}\quad m_1\mbox{-a.e.}
\end{equation}
By regularity of the forms $B^{(t)}$, for every $a>0$ there exists $\eta^{\varepsilon,a}_n\in \mathcal C$  such that
\begin{equation}
\label{eq.cl2}
\|J_n(\eta^\varepsilon_n)-\eta^{\varepsilon,a}_n\|_{\WW^2}\le a
\end{equation}
(see \cite[Lemma 1.1]{O3}). Let $\xi_n\in\mathcal C$ be a positive function such that $\xi_n=1$ on $K$ and $\mbox{supp}[\xi_n]\subset U_{n-1}$. The existence of $\xi_n$
follows from \cite[Lemma 1.1]{O3}. We put
\[
\bar\eta^{\varepsilon,a}_n=\eta^{\varepsilon,a}_n+(\xi_n-\eta^{\varepsilon,a}_n)^+.
\]
Observe that $\bar\eta^{\varepsilon,a}_n\in\mathcal C$ and $\bar\eta^{\varepsilon,a}_n\ge\mathbf{1}_K$.
Let $a_l\downarrow 0$. By (\ref{eq.cl2}), up to a subsequence, $(\xi_n-\eta^{\varepsilon,a_l}_n)^+\rightarrow (\xi_n-J_n(\eta^\varepsilon_n))^+$ weakly in $\WW^2$ as $l\rightarrow\infty$. Hence, up to a subsequence,
\begin{equation}
\label{eq.cl3}
\frac1{k}\sum_{l=1}^k(\xi_n-\eta^{\varepsilon,a_l}_n)^+\rightarrow (\xi_n-J_n(\eta^\varepsilon_n))^+
\end{equation}
strongly in $\WW^2$. Furthermore,
\[
\mbox{CAP}(K)\le \|\frac1{k}\sum_{l=1}^k\bar\eta^{\varepsilon,a_l}_n\|_\WW\le \|\frac1{k}\sum_{l=1}^k\eta^{\varepsilon,a_l}_n\|_\WW
+\|\frac1{k}\sum_{l=1}^k(\xi_n-\eta^{\varepsilon,a_l}_n)^+\|_{\WW^2}.
\]
By (\ref{eq.cl1}), $(\xi_n-J_n(\eta^\varepsilon_n))^+=0$, so letting $k\rightarrow\infty$
in the above inequality and using (\ref{eq.cl2}) and  (\ref{eq.cl3}) we get
\[
\mbox{CAP}(K)\le  \|J_n(\eta^\varepsilon_n)\|_\WW\le \|\eta^\varepsilon_n\|_\WW\le b_n+\varepsilon.
\]
Hence $\mbox{CAP}(K)=0$. Since $\mbox{c}_2$ is a Choquet capacity, this implies that  $\mbox{CAP}(A)=0 $ if $\mbox{c}_2(A)=0$.
\end{proof}

\section{Smooth measures and associated additive functionals}
\label{sec4}

The results of this section will be needed in Sections \ref{sec7} and \ref{sec8}.
Recall that a positive additive functional (AF in abbreviation) $A^1$ of $\BX^1$  is said to be in the Revuz correspondence with a positive $\mu\in S^1(E^1)$ if
\begin{equation}
\label{rd.1}
\lim_{\alpha\rightarrow\infty}\alpha
E^1_{m_1}\int^{\infty}_0e^{-\alpha t}
f(X^1_t)\,dA^1_t=\int_{E^1}f(x)\,\mu(dx),\quad f\in\BB_b^+(E^1),
\end{equation}
where $E^1_{m_1}$  denotes the expectation with respect to the measure
$P^1_{m_1}$ defined as $P^1_{m_1}
(\cdot)=\int_{E^1}P_x(\cdot)\,m_1(dx)$. Similarly, we say that a positive AF $A$
of $\BX$ is in the Revuz correspondence with a positive $\mu\in S(E_{0,T})$ if
\begin{equation}
\label{rd.2}
\lim_{\alpha\rightarrow\infty}\alpha
E_{m_1}\int^{\infty}_0e^{-\alpha t}
f(X_t)\,dA_t=\int_{E_{0,T}}f(x)\,\mu(dx),\quad f\in\BB_b^+(E_{0,T}),
\end{equation}
where  $E_{m_1}$ denotes the expectation with respect to
$P_{m_1}(\cdot)=\int_{E_{0,T}}P_x(\cdot)\,m_1(dx)$.
%(see \cite{O2} or ? for more details).

Let $A^f_t=\int_0^t f(X_r)\,dA_r$. We have
\begin{equation}
\label{eq7.5ccc}
\lim_{\alpha\rightarrow\infty}\alpha^2
E_{m_1}\int^{\infty}_0e^{-\alpha t} A^f_t\,dt=\lim_{\alpha\rightarrow\infty}
\int^{\infty}_0se^{-s}\frac{\alpha}{s}E_{m_1}A^f_{s/\alpha}\,ds.
\end{equation}
Since $t\mapsto E_{m_1}A^f_t$ is subadditive, $(1/t)E_{m_1}A^f_t$ increases as $t$ decreases, and  moreover, $\lim_{t\downarrow0}(1/t)E_{m_1}A^f_t=\sup_{t>0}(1/t)E_mA^f_t$. Therefore, letting $\alpha\rightarrow\infty$ in (\ref{eq7.5ccc}), shows that
\begin{align*}
\lim_{\alpha\rightarrow\infty}\alpha
E_{m_1}\int^{\infty}_0e^{-\alpha t}\,dA^f_t&=\lim_{\alpha\rightarrow\infty}\alpha^2
E_{m_1}\int^{\infty}_0e^{-\alpha t}A^f_t\,dt\\
&=\sup_{t>0}\frac{1}{t}E_{m_1}A^f_t
\cdot\int^{\infty}_0se^{-s}\,ds= \sup_{t>0}\frac{1}{t}E_{m_1}A^f_t.
\end{align*}
A similar  result holds for $A^1$.

It is known (see \cite[Section 2]{K:JFA}) that for every $\mu\in
S^1(E^1)$ there exists a unique positive  natural AF $A^1$ of $\BX^1$
(i.e. a positive AF of $\BX^1$ such that $A^1$ and $\BX^1$ have no common
discontinuities) such that $A^1$ is in the Revuz correspondence with
$\mu$. In what follows we  denote it by $A^{1,\mu}$. In fact,
$A^{1,\mu}$ is a predictable process (see \cite[Theorem 5.3]{GS}). In the proposition  below, we show a similar result for positive smooth measures on $E_{0,T}$.

\begin{proposition}
\label{prop4.1}
Let $\mu\in S(E_{0,T})$ be positive.
There exists a unique positive natural AF $A^\mu$ of $\BX$ in the Revuz correspondence
with $\mu$. Moreover, for q.e. $x\in E_{0,T}$,
\begin{equation}
\label{eq2.rede}
E_x\int^{\infty}_0e^{-\alpha t}
f(X_t)\,dA^{\mu}_t=E^1_x\int^{\infty}_0e^{-\alpha t}
f(X^1_t)\,dA^{1,\bar\mu}_t,\quad \alpha\ge 0,\quad f\in\BB^+_b(E_{0,T}),
\end{equation}
where $\bar\mu$ denotes the extension of $\mu$ to $E^1$ such that $\mu(E^1\setminus E_{0,T})=0$ and $A^{1,\bar\mu}$ is the positive natural AF of $\BX^1$ in the Revuz correspondence with $\bar\mu$.
\end{proposition}
\begin{proof}
We first assume that $\mu\in S_0(E_{0,T})$. Since $\mu$ is zero outside $E_{0,T}$, by  Lemma \ref{lm.eqc12} it is clear that $\bar\mu\in S_0(E^1)$.
By the remark preceding Proposition \ref{prop4.1}, there  exists a positive natural AF $A^{1,\bar\mu}$ of $\BX^1$ in the Revuz correspondence with $\bar\mu$. Let $R^1\bar\mu(x)=E^1_xA^{1,\bar\mu}_\zeta$. By (\ref{rd.1}),
\begin{equation}
\label{rd.3}
\mathbf{1}_{E_{0,T}}R^1\bar\mu=R^1\bar\mu
\end{equation}
for $x\in E_{0,T}$. Hence, by \cite[Theorem 9.3]{GS} and (\ref{eq8.c17})
(for the dual process), for every positive $\eta\in\BB(E^1)\cap L^2(E^1;m_1)$ such that $\eta=0$ outside $E_{0,T}$ we have
\begin{align*}
\int_{E^1}R^1\bar\mu\,\eta\,dm_1=\int_{E^1}\hat R^1\eta\,d\bar\mu
&=\int_{E_{0,T}}\hat R^1\eta\,d\bar\mu\\
&=\int_{E_{0,T}}\hat R\eta\,d\mu
=\EE(U\mu,\hat R\eta)=\int_{E_{0,T}} U\mu\,\eta\,dm_1.
\end{align*}
Thus $R^1\bar\mu=U\mu$ $m_1$-a.e. on $E_{0,T}$. By (\ref{eq2.cpr}) and (\ref{rd.3}),
for all $\tau\le\zeta$ and $x\in E_{0,T}$,
\[
E_x(R^1\bar\mu)(X_\tau)=E^1_x(R^1\bar\mu)(X^1_\tau\circ\Pi)=E^1_x(R^1\bar\mu)(X^1_\tau)\le E^1_x(A^{1,\bar\mu}_\zeta-A^{1,\bar\mu}_\tau).
\]
Hence $R^1\bar\mu$ is a natural potential with respect to $\mathbb X$.
By \cite[Theorem IV.4.22]{BG}, there exists a unique positive natural AF $A$ of $\mathbb X$ such that
\begin{equation}
\label{eq8.c18}
R^1\bar\mu=E_\cdot A_\zeta\quad \mbox{q.e. on } E_{0,T}.
\end{equation}
From this and the fact that $R^1\bar\mu=U\mu$ $m_1$-a.e. on $E_{0,T}$ we get
\[
\lim_{t\rightarrow 0}\frac1tE_{\eta\cdot m_1}A_t=(R^1\bar\mu,\frac1t(\eta-T_t\eta))_\HH=(U\mu,\frac1t(\eta-T_t\eta))_\HH=
\EE(U\mu,\eta)=\int_{E_{0,T}}\eta\,d\mu
\]
for every $\eta\in\WW_T\cap C_b(E_{0,T})$. Thus $\mu$ is a Revuz measure of $A$.
%By \cite[Theorem 9.3]{GS} $\mu$ is the Revuz measure of $A$.
Let $A^\mu:=A$. From (\ref{eq8.c18}) we easily get  (\ref{eq2.rede}).

To prove the existence of $A^{\mu}$ in the general case, we choose
an $\EE$-nest $\{F_n\}$ of compact subsets of $E_{0,T}$ such that
$\mu_n=\mathbf{1}_{F_n}\cdot\mu\in S_0(E_{0,T})$. Such a nest exists by \cite[Theorem 4.7]{T2}. By what has already been proved,  for all $\alpha\ge 0$ and $f\in\BB^+_b(E_{0,T})$,
\begin{equation}
\label{eq2.redesob}
E_x\int^{\infty}_0e^{-\alpha t}f(X_t)\,dA^{\mu_n}_t =E^1_x\int^{\infty}_0e^{-\alpha t}
f(X^1_t)\,dA^{1,\bar\mu_n}_t
\end{equation}
for q.e. $x\in E_{0,T}$.
Since the additive functional in the Revuz correspondence with a smooth measure is uniquely determined,  we have $dA^{\mu_n}\le dA^{\mu_m}$ for $n\le m$ and
$A^{1,\bar\mu_n}_t=\int_0^t\mathbf{1}_{F_n}(X_r)\,dA^{1,\bar\mu}_r$, $t\ge 0$, $P^1_x$-a.s. We now set
\begin{equation}
\label{eq4.5}
A^{\mu}_t=\sup_{n\ge 1}A^{\mu_n}_t,\quad t\ge 0.
\end{equation}
By (\ref{eq2.redesob}),
\begin{align*}
\lim_{\alpha\rightarrow\infty}\alpha
E_{m_1}\int^{\infty}_0e^{-\alpha t} f(X_t)\,dA^\mu_t
&=\lim_{n\rightarrow \infty}\lim_{\alpha\rightarrow\infty}\alpha E^1_{m_1}\int^{\infty}_0e^{-\alpha t}\mathbf{1}_{F_n}
f(X^1_t)\,dA^{1,\bar\mu}_t\nonumber\\
&=\int_{\bigcup_{n\ge 1}F_n}f\,d\bar\mu=\int_{E_{0,T}}f\,d\mu,
\end{align*}
so $A^{\mu}$ defined by (\ref{eq4.5}) is in the Revuz correspondence with $\mu$. By \cite[Lemma 1, page 182]{M},
$A^\mu$ is a positive natural AF of $\mathbb X$. By (\ref{eq2.cpr}) and \cite[Remark IV.3.6]{St2}, for $m_1$-a.e. $x\in E_{0,T}$,
\begin{align*}
0&=P_x(\lim_{n\rightarrow \infty}\sigma_{E_{0,T}\setminus F_n}<\zeta)\\
&=P_x\Big(\omega\in\Omega: \lim_{n\rightarrow \infty}\inf_{t>0}\{X_t(\omega)\in E_{0,T}\setminus F_n\}<\zeta(\omega)\Big)\\
&=P^1_x\circ\Pi^{-1}\Big(\omega\in\Omega: \lim_{n\rightarrow \infty}\inf_{t>0}\{X_t(\omega)\in E_{0,T}\setminus F_n\}<\zeta(\omega)\Big)\\
&=P^1_x\Big(\omega\in\Omega^1:\lim_{n\rightarrow \infty}\inf_{t>0}\{X^1_t(\Pi(\omega))\in E_{0,T}\setminus F_n\}<\zeta(\Pi(\omega))\Big)\\
&=P^1_x\Big(\omega\in\Omega^1:\lim_{n\rightarrow \infty}\inf_{t>0}\{X^1_t(\omega)\in E_{0,T}\setminus F_n\}<\zeta(\omega)\Big).
\end{align*}
Therefore, letting $n\rightarrow\infty$ in (\ref{eq2.redesob}), yields (\ref{eq2.rede}).
\end{proof}
%Conversely,
%if $A$ is a positive natural AF of $\BX$, then by
%Proposition in Section II.1 of \cite{R} and \cite[Theorem 5.6]{O2},
%there exists a smooth measure $\mu$ on $E^1$ such that $A$ is in
%the Revuz correspondence with $\mu$.
\begin{remark}
Since the capacities $\mbox{Cap}_{\psi}$ associated with $\EE$ and the dual form $\hat \EE$ coincide (see Section \ref{sec2.1}), the set  of smooth measures $S(E_{0,T})$ associated with $\EE$ coincides with the set of smooth measures on $E_{0,T}$ associated with the dual form $\hat\EE$. Therefore from Proposition \ref{prop4.1} applied to $\hat\EE$ and $\hat\BX$ it follows that for a positive $\mu\in S(E_{0,T})$ there exists a unique predictable AF of $\hat\BX$ in the Revuz correspondence with $\mu$. We will denote it by $\hat A^{\mu}$.
\end{remark}

In what follows, for a  positive smooth measure $\mu$ on $E_{0,T}$, we put
\[
R_\alpha\mu(x)=E_x\int_0^{\zeta}e^{-\alpha t}\,dA^\mu_t,
\qquad\hat R_\alpha\mu(x)=\hat E_x\int_0^{\zeta}e^{-\alpha t}\,d\hat A^\mu_t,\quad x\in E_{0,T}.
\]
By \cite[Theorem 9.3]{GS}, for every $\mu\in\hat S_0(E_{0,T})$ we have
\begin{equation}
\label{eq.sob7}
\hat U_\alpha\mu(x)=\hat R_{\alpha}\mu(x)
\end{equation}
for   $m_1$-a.e. $x\in E_{0,T}$, and for every $\mu\in S_0(E_{0,T})$,
\begin{equation}
\label{eq4.8}
U_\alpha\mu(x)=R_{\alpha}\mu(x)
\end{equation}
for   $m_1$-a.e. $x\in E_{0,T}$.

\begin{lemma}
\label{lem4.2}
Let $\alpha\ge0$. For any positive  $\mu\in S(E_{0,T})$ and $\eta\in\BB^+(E_{0,T})$,
\begin{equation}
\label{eq2.2}
\int_{E_{0,T}}R_\alpha\eta\,d\mu=\int_{E_{0,T}}\eta\hat R_\alpha\mu\,dm_1,\qquad
\int_{E_{0,T}}\hat R_\alpha\eta\,d\mu=\int_{E_{0,T}}\eta R_\alpha\mu\,dm_1.
\end{equation}
\end{lemma}
\begin{proof}
First we assume that $\mu\in\hat S_0(E_{0,T})$ and $\eta\in\BB^+(E_{0,T})\cap\HH$. Then, by (\ref{eq2.22}), the fact that $R_{\alpha}\eta$ is an $m_1$-version of $G_{\alpha}\eta$ and (\ref{eq2.26}),
\[
\int_{E_{0,T}}R_\alpha\eta \,d\mu=\EE_\alpha(R_\alpha\eta,\hat U_{\alpha}\mu)
=\EE_\alpha(G_\alpha\eta,\hat U_\alpha\mu )=\int_{E_{0,T}}\eta \hat U_\alpha\mu\,dm_1.
\]
From this  and (\ref{eq.sob7}) it follows that the first equality in (\ref{eq2.2}) is satisfied. To show it  in the general case, we set
$\eta_n=\frac{ng}{1+ng}(\eta\wedge n)$ for some strictly positive $g\in\HH$ and choose
an $\EE$-nest $\{F_n\}$ of compact sets such that $\mu_n=\mathbf{1}_{F_n}\cdot\mu\in \hat S_0(E_{0,T})$. Such a nest exists by \cite[Theorem 4.7]{T2}. By what has already been proved,
\[
\int_{E_{0,T}}R_\alpha\eta_n\,d\mu_n=\int_{E_{0,T}}\eta_n\hat R_\alpha\mu_n\,dm_1,\quad n\ge1.
\]
Letting $n\rightarrow\infty$  and using \cite[Remark 3.6]{St2} yields the first equality in (\ref{eq2.2}). The proof of the second equality is similar, so we omit it.
\end{proof}

Given a positive smooth measure $\mu$ on $E_{0,T}$, we denote by $R_\alpha \circ\mu$ the Borel measure on $E_{0,T}$ defined by the formula
\[
\int_{E_{0,T}}\eta\,d(R_\alpha\circ\mu)
=\int_{E_{0,T}}R_\alpha\eta\,d\mu,\quad\eta\in\BB^+(E_{0,T}),
\]
and by $\hat R_{\alpha}\circ\mu$ we denote the  Borel measure on $E_{0,T}$ defined by
\[
\int_{E_{0,T}}\eta\,d(\hat R_\alpha\circ\mu)
=\int_{E_{0,T}}\hat R_\alpha\eta\,d\mu,\quad\eta\in\BB^+(E_{0,T}).
\]

By Lemma \ref{lem4.2}, the measures $R_\alpha\circ\mu$,  $\hat R_\alpha\circ\mu$ are absolutely continuous with respect to $m_1$ and
\[
R_\alpha\circ\mu= (\hat R_\alpha\mu)\cdot m_1,\qquad \hat R_\alpha\circ\mu= (R_\alpha\mu)\cdot m_1.
\]

\section{Decomposition of smooth measures}
\label{sec5}

In what follows, we denote by $\langle\langle\cdot,\cdot\rangle\rangle$ the duality pairing between $\WW'_0$ and $\WW_0$.

In this section, we prove that each $\mu\in\MM_b(E_{0,T})$ admits decomposition (\ref{eq1.4}). The converse statement will be proved in Section \ref{sec7}.
We start with  a decomposition of elements of the space $\WW'_0$.

\begin{proposition}
\label{prop5.1} Let $g\in\WW_0'$. Then there exist $g_1\in\VV'$,
$g_2\in\VV$ such that
\begin{equation}
\label{eq3.15} \|g_1\|_{\VV'}\le\|g\|_{\WW'_0}, \quad
\|g_2\|_{\VV}\le \|g\|_{\WW'_0}
\end{equation}
and for every $v\in\WW_0$,
\[
\langle\langle g,v\rangle\rangle=\langle g_1,v\rangle
-\langle\partial_tv,g_2\rangle.
\]
\end{proposition}
\begin{proof}
The proof is a slight modification of the proof of \cite[Lemma
2.24]{DPP}. We give it for completeness. Let $F=\VV\times\VV'$,
$\|(v_1,v_2)\|_F=\|v_1\|_{\VV}+\|v_2\|_{\VV'}$, and let
\[
T:\WW_0\rightarrow F,\quad T(u)=(u,\partial_tu).
\]
Since $\|T(u)\|_F=\|(u,\partial_tu)\|_F=\|u\|_{\WW}$, $T$ is an isometry.
We equip $T(\WW_0)$ with the norm of $F$ and define
\[
\Phi_g:T(\WW_0)\rightarrow\BR,\qquad \Phi_g(v_1,v_2)=\langle\langle
g,T^{-1}(v_1,v_2)\rangle\rangle,
\]
i.e. $\Phi_g(T(u))=\langle\langle g,u\rangle\rangle$ for $u\in\WW_0$.
Then $\Phi_g$ is a continuous linear functional on $T(\WW_0)$.
%K\\
%Istotnie, poniewaz $\Phi_g$ jest liniowe, to wystarczy sprawdzic
%ciaglosc w zerze. Przypuscmy, ze
%$T(\WW_0)\ni(v^n_1,v^n_2)\rightarrow(0,0)$ w $E$. Niech
%$u_n\in\WW_0$ bedzie takie, ze $T(u_n)=(v^n_1,v^n_2)$. Poniew/z
%$T$ jest izometria, to $u_n\rightarrow0$ w $\WW_0$. Dlatego
%$\Phi_g(v^n_1,v^n_2)=\langle\langle
%g,u_n\rangle\rangle\rightarrow0$.\\
%KK\\
By the  Hahn-Banach theorem, $\Phi_g$ extends to a continuous linear  functional on $F$ (still denoted by $\Phi_g$) such that
\begin{equation}
\label{eq3.16} \|\Phi_g\|_{F'}=\|g\|_{\WW'_0}\,.
\end{equation}
%K\\
%Mamy
%\[
%\|\Phi_g\|=\sup_{(v_1,v_2)\in T(\WW_0):\|(v_1,v_2)\|_F\le1}
%|\langle\langle g,T^{-1}(v_1,v_2)\rangle\rangle|
%\le\sup_{u\in\WW_0,\|u\|_{\WW}\le1}
%|\langle\langle g,u\rangle\rangle|=\|g\|_{\WW'}\,.
%\]
%(przedostatnia rownosc wynika z faktu, ze $T$ jest izometria).
%\\
%KK\\
Since  $\VV$ is reflexive, $F'=\VV'\times\VV$, so $\Phi_g$ is of the form
\[
\Phi_g(v_1,v_2)=\langle g_1,v_1\rangle -\langle v_2,g_2\rangle,\quad
(v_1,v_2)\in F,
\]
for some $g_1\in\VV'$, $g_2\in\VV$.
%\\
%K\\
%Mamy $\Phi_g\in F'=(\VV\times\VV')'=\VV'\times\VV''(=\VV'\times\VV)$.
%Dlatego $\Phi_g(v_1,v_2)=\langle g_1,v_1\rangle+\langle g''_2,v_2\rangle$ %dla pewnych $g_1\in\VV'$, $g''_2\in\VV''$. Poniewaz $\VV$ refleksywna, to %$\langle g''_2,v_2\rangle_{\VV'',\VV'}=\langle v_2,g_2\rangle$ dla pewnego %$g_2\in\VV$.
%Rozpartrujac $-g_2$ zamiast $g_2$ mozemy miec znak ``$-$".\\
%KK\\
One can check that $\|g_1\|_{\VV'}\le\|\Phi_g\|_{F'}$ and
$\|g_2\|_{\VV}\le \|\Phi_g\|_{F'}$, which when combined with
(\ref{eq3.16}) yields (\ref{eq3.15}).
%\\
%K:\\
%Mamy
%\begin{align*}
%\|g_1\|_{\VV'}&=\sup_{v_1\in\VV,\|v_1\|_{\VV}\le1}|\langle g_1,v_1\rangle|
%=\sup_{v_1\in\VV,\|v_1\|_{\VV}\le1,v_2=0}|\langle g_1,v_1\rangle-\langle %v_2,g_2\rangle|\\
%&\le\sup_{(v_1,v_2)\in F,\|(v_1,v_2)\|_{F}\le1}|\langle g_1,v_1\rangle %-\langle v_2,g_2\rangle|=\|\Phi_g\|_{F'}\le\|g\|_{\WW'}
%\end{align*}
%oraz
%\begin{align*}
%\|g_2\|_{\VV}&=\|g''_2\|_{\VV''}=\sup_{v_2\in\VV',\|v_2\|_{\VV'}\le1}|\langle %g''_2,v_2\rangle|=\sup_{v_1=0,v_2\in\VV',\|v_2\|_{\VV'}\le1}
%|\langle g_1,v_1\rangle+\langle g''_2,v_2\rangle|\\
%&\le\sup_{(v_1,v_2)\in F,\|(v_1,v_2)\|_{F}\le1}
%|\langle g_1,v_1\rangle+\langle g''_2,v_2\rangle|\\
%&=\sup_{(v_1,v_2)\in F,\|(v_1,v_2)\|_{F}\le1}
%|\langle g_1,v_1\rangle+\langle v_2, %g_2\rangle|=\|\Phi_g\|_{F'}\le\|g\|_{\WW'}.
%\end{align*}
%KK\\
Hence
\[
\langle\langle g,v\rangle\rangle =\Phi_g(T(v))=\Phi_g(v,\partial_tv)=\langle
g_1,v\rangle -\langle\partial_tv,g_2\rangle
\]
for $v\in\WW_0$, and the lemma is proved.
\end{proof}

\begin{lemma}
\label{lem2.2}
If  $v\in\WW_0$, then
\begin{equation}
\label{eq2.3} \alpha\hat G_{\alpha}v\rightarrow v\mbox{ weakly in } \WW_0.
\end{equation}
\end{lemma}
\begin{proof}
By \cite[Proposition I.III.7]{St2}, for every $w\in\VV$,
\begin{equation}
\label{eq3.5} \alpha G_{\alpha}w\rightarrow  w,\quad  \alpha\hat G_{\alpha}w\rightarrow  w\quad\mbox{in }
\VV.
\end{equation}
By (\ref{eq3.5}), for $w\in\VV$ we have
\begin{align*}
\langle\hat\Lambda(\alpha\hat
G_{\alpha}v),w\rangle&=\BB(w,\alpha\hat G_{\alpha}v)
-\EE(w,\alpha\hat G_{\alpha}v)\\&
=\BB(w,\alpha\hat G_{\alpha}v)
-\EE(\alpha G_{\alpha}w,v)\rightarrow\BB(w,v) -\EE(w,v)
=\langle\hat\Lambda v,w\rangle.
\end{align*}
Since $\VV$ is reflexive, $\{\hat\Lambda(\alpha\hat
G_{\alpha}v)\}$ is weakly convergent in $\VV'$ as
$\alpha\rightarrow\infty$. From this, (\ref{eq3.5}) and Proposition
\ref{prop5.1} we get (\ref{eq2.3}).
\end{proof}

\begin{proposition}
\label{prop3.2} Let $\mu\in\MM_{0,b}(E_{0,T})$. Then there exist $f\in
L^1(E_{0,T};m_1)$ and $g\in\WW'_0$ such that
\begin{equation}
\label{eq3.decb}
\mu=f\cdot m_1+g,
\end{equation}
i.e. for every bounded $v\in\WW_0$ we have
\begin{equation}
\label{eq3.6} \int_{E_{0,T}} \tilde v\,d \mu=\int_{E_{0,T}}fv\,dm_1+\langle\langle
g,v\rangle\rangle.
\end{equation}
\end{proposition}
\begin{proof} Let $\mu=\mu^+-\mu^-$ be the Hahn decomposition of $\mu$. Then $\mu^+,\mu^-\in\MM_{0,b}(E_{0,T})$. Therefore, without loss of generality, we can assume that $\mu$ is positive. By \cite[Theorem 4.7]{T2},  there exists a nest $\{F_n\}$ such that
$\fch_{F_n}\cdot|\mu|\in S_0(E_{0,T})$ for every $n\in\BN$. Set
$\mu_n=\fch_{F_{n+1}\setminus F_n}\cdot\mu$. Then $|\mu_n|\in S_0(E_{0,T})$
and
\begin{equation}
\label{eq3.10} \sum^{\infty}_{n=1}\mu_n=\mu
\end{equation}
because the set $(\bigcup^{\infty}_{n=1}F_n)^c$ is exceptional.
Let $\mu^{\alpha}_n=\alpha \hat R_{\alpha}\circ\mu_n$. Then for $v\in\WW_0$
we have
\[
\EE_1(U_1\mu_n^{\alpha},v)=\int_{E_{0,T}} \tilde v\,d\mu^{\alpha}_n=\int_{E_{0,T}}\alpha
\hat R_{\alpha}v\,d\mu_n,\qquad \EE_1(U_1\mu_n,\alpha\hat R_{\alpha}v)
=\int_{E_{0,T}}\alpha\hat R_{\alpha}v\,d\mu_n.
\]
Hence
\[
\EE_1(U_1\mu^{\alpha}_n,v) =\EE_1(U_1\mu_n,\alpha\hat
R_{\alpha}v),
\]
so by  Lemma \ref{lem2.2},
\begin{equation}
\label{eq3.3}
\EE_1(U_1\mu^{\alpha}_n,v)\rightarrow\EE_1(U_1\mu_n,v).
\end{equation}
Let $j:\VV\rightarrow\WW'_0$ be the mapping defined as
\[
\langle\langle j(u),v\rangle\rangle=\EE_1(u,v),\quad
u\in\VV,v\in\WW_0.
\]
From (\ref{eq3.3}) and reflexivity of the space $\WW_0$ it follows
that the sequence $\{j(U_1\mu^{\alpha}_n)\}$ converges to
$j(U_1\mu_n)$ weakly in $\WW'_0$  as $\alpha\rightarrow\infty$.
%\\
%K:\\
%Niech $u_n'\rightarrow u'$ s/labo w $\WW'_0$. Wtedy, z definicji,
%\[
%\langle u'',u_n'\rangle\rightarrow\langle u'',u'\rangle\qquad (*)
%\]
%dla dowolnego $u''\in(\WW_0')'=\WW_0''$. Poniewa/z $\WW_0$ jest
%refleksywna, to kanoniczne w/lo/zenie $J:\WW_0\rightarrow\WW_0''$
%jest ,,na''. Dlatego dla ka/zdego $u''$ istnieje $u$ takie, /ze
%$u''=J(u)$. Z $(*)$ wynika wi/ec, /ze dla ka/zdego $u\in\WW_0$,
%\[
%\langle u'_n,u\rangle=\langle J(u),u_n'\rangle\rightarrow \langle
%J(u),u'\rangle=\langle u',u\rangle,
%\]
%Na odwr/ot, przypu/s/cmy, /ze $u_n',u'\in\WW'$ oraz $\langle
%u_n',u\rangle\rightarrow\langle u',u\rangle$ dla ka/zdego
%$u\in\WW_0$. Wtedy
%\[
%\langle J(u),u_n'\rangle=\langle u_n',u\rangle\rightarrow\langle
%u',u\rangle=\langle J(u),u'\rangle,
%\]
%a wi/ec, poniewa/z $J$ jest ,,na'',  zachodzi $(*)$ dla ka/zdego
%$u''\in\WW_0''$.\\
%KK:\\
By the Banach-Saks theorem, there is a sequence $\{\alpha_l\}$
such that $\alpha_l\rightarrow\infty$ as $l\rightarrow\infty$ and
\[
j(U_1\hat f_k(\mu_n))\rightarrow j(U_1\mu_n)\quad\mbox{in }
\WW'_0
\]
as $k\rightarrow\infty$, where
\[
\hat f_k(\mu_n)=\frac1k\sum^k_{l=1}\mu^{\alpha_l}_n.
\]
Therefore one can find a subsequence $\{k_n\}$ such that for every
$n\in\BN$,
\begin{equation}
\label{eq3.1} \|j(U_1(\mu_n-\hat f_{k_n}(\mu_n)))\|_{\WW'_0}\le
2^{-n}.
\end{equation}
By (\ref{eq2.2}), $\hat f_k(\mu_n)=f_k(\mu_n)\cdot m_1$, where
\[
f_k(\mu_n)=\frac1k\sum^k_{l=1}\alpha_lU_{\alpha_l}\mu_n.
\]
Set
\begin{equation}
\label{eq3.11}
g=\sum^{\infty}_{n=1}j(U_1(\mu_n-(\hat f_{k_n}(\mu_n))),\qquad
f=\sum^{\infty}_{n=1}f_{k_n}(\mu_n).
\end{equation}
By (\ref{eq3.1}), $g\in\WW'_0$. Since $m_1$ is $\sigma$-finite,
there exists a sequence $\{V_i\}$ of Borel subsets of $E_T$ such
that $\bigcup^{\infty}_{i=1}V_i=E_T$, $V_i\subset V_{i+1}$ and
$m_1(V_i)<\infty$, $l\in\BN$. By (\ref{eq4.8}) and Lemma \ref{lem4.2},
for every $i\in\BN$ we have
\begin{align}
\label{eq.intd}
\int_{E_{0,T}}\mathbf{1}_{V_i}f\,dm_1= \sum_{n=1}^{\infty}
\int_{E_{0,T}}\mathbf{1}_{V_i}f_{k_n}(\mu_{n})\,dm_1=\sum_{n=1}^{\infty}
\int_{E_{0,T}}\frac{1}{k_n}\sum_{l=1}^{k_n}\alpha_l \hat R_{\alpha_l}(\mathbf{1}_{V_i})\,d\mu_{n}.
\end{align}
Since $\alpha\hat R_{\alpha}$ is a Markov operator, it follows that
\[
\int_{E_{0,T}}\mathbf{1}_{V_i}f\,dm_1\le\sum_{n=1}^{\infty}\|\mu_{n}\|_{TV}=\|\mu\|_{TV}
\]
Hence $\|f\|_{L^1(E_T;m_1)}<\infty$ by the monotone convergence
theorem. Since
\[
\int_{E_{0,T}} \tilde v\,d \mu_n=\EE_1(U_1\mu_n,v)
\]
and
\[
\int_{E_{0,T}}f_{n_k}(\mu_n)v\,dm_1=\int_{E_{0,T}}
\tilde v\,d \hat f_{k_n}(\mu_n)
=\EE_1(U_1(\hat f_{k_n}(\mu_n)),v),
\]
for every  $N\in\BN$ we have
\begin{align*}
\sum^{N}_{n=1}\int_{E_{0,T}}\tilde v\,d\mu_n
&=\sum^{N}_{n=1}\EE_1(U_1(\mu_n-\hat f_{k_n}(\mu_n)),v) +\sum^{N}_{n=1}\int_{E_{0,T}}f_{k_n}(\mu_n)v\,dm_1\\
&=\langle\langle\sum^N_{n=1}j(U_1(\mu_n-\hat f_{k_n}(\mu_n))),v\rangle\rangle+\sum^N_{n=1}\int_{E_{0,T}}(f_{k_n}(\mu_n)v\,dm_1.
\end{align*}
Letting $N\rightarrow\infty$ and using (\ref{eq3.10}),
(\ref{eq3.11}) we obtain (\ref{eq3.6}), which completes the proof of the proposition.
\end{proof}

\begin{remark}
The decomposition (\ref{eq3.decb}) also holds for arbitrary $\mu\in S(E_{0,T})$.
In this case,  in general, $f\notin L^1(E_{0,T};m_1)$ but $f$  is quasi-integrable, i.e. there exists a nest $\{F_n\}$
of compact subsets of $E_{0,T}$ such that $\mathbf{1}_{F_n}\cdot f\in L^1(E_{0,T};m_1)$,
$n\ge1$. The proof of the decomposition (\ref{eq3.decb}) in case $\mu\in S(E_{0,T})$ runs as the proof of Proposition \ref{prop3.2}
with the only difference that in (\ref{eq.intd}) we replace $V_i$ by
$F_i$  defined in the proof of Proposition \ref{prop3.2}. We then get
\[
\int_{E_{0,T}}\mathbf{1}_{F_i}f\,dm_1\le \|\mathbf{1}_{F_i}\cdot\mu\|_{TV},\quad i\ge 1.
\]
\end{remark}

Combining Proposition \ref{prop5.1} with Proposition \ref{prop3.2} we get
the following theorem.

\begin{theorem}
\label{th3.2} Let $\mu\in\MM_{0,b}$. Then there exist $f\in L^1(E_{0,T};m_1)$, $g_1\in\VV'$, $g_2\in\VV$ such that
\begin{equation}
\label{eq5.12}
\int_{E_{0,T}}\tilde v\,d\mu=\int_{E_{0,T}}fv\,dm_1+\langle g_1,v\rangle
-\langle\partial_tv,g_2\rangle
\end{equation}
for every bounded $v\in\WW_0$.
\end{theorem}

Note that $f,g_1,g_2$ in Theorem \ref{th3.2} are not uniquely determined. In what follows, any triplet of functions having the same properties as the triplet $(f,g_1,g_2)$ appearing in Theorem \ref{th3.2} will be called a decomposition of $\mu$.

\begin{remark}
From (\ref{eq3.11}) it follows that if $\mu\in\MM_{0,b}(E_{0,T})$ is positive, then the $L^1$ part $f$ of the decomposition of $\mu$ can be chosen to be positive.
\end{remark}

\section{Semilinear equations with right-hand side in $\WW'_0$}
\label{sec6}

In this section, using the decomposition of elements of $\WW'_0$ given in Proposition \ref{prop5.1}, we will show an existence and uniqueness result for the following Cauchy
problem
\begin{equation}
\label{eq7.1}
-\partial_t u-L_t u=f(\cdot,u)+g,\qquad u(T,\cdot)=\varphi.
\end{equation}
In (\ref{eq7.1}),  $g\in\WW'_0$, $f:E_{0,T}\times\BR\rightarrow \BR$ and $\varphi\in L^2(E;m)$.

We will need the following assumption: there exist $\lambda\in\BR$, $M\ge 0$ and a positive $\varrho\in \HH$ such for all $x\in E_{0,T}$ and $y,y'\in\BR$,
\begin{equation}
\label{eq6.2} |f(x,y)|\le \varrho(x)+M|y|,\qquad (f(x,y)-f(x,y'))(y-y')\le \lambda |y-y'|^2.
\end{equation}

\begin{definition}
We say that $u\in\VV$ is a solution of (\ref{eq7.1}) if
\begin{equation}
\label{eq7.2} \langle\partial_t \eta,u\rangle+\BB(u,\eta)
=(\varphi,\eta(T,\cdot))_{H}+(f(\cdot,u),\eta)_{\HH}+\langle\langle
g,\eta\rangle\rangle,\quad \eta\in\WW_0,
\end{equation}
where $\BB$ is defined by (\ref{eq2.24}).
\end{definition}

In what follows, for a function $\eta$ on $E_{0,T}$ and
$\varepsilon>0,\, t\in [0,T]$, we set $\eta^t=
\eta\mathbf{1}_{[t,T]\times E}$ and
\[
\eta^t_\varepsilon(s)=0,\, s\in [0,t], \quad
\eta^t_\varepsilon(s)=\varepsilon^{-1}\eta(t+\varepsilon)(s-t),\,
t\in [t,t+\varepsilon],\quad\eta^t_\varepsilon(s)= \eta(s),\, s\in
[t+\varepsilon,T].
\]
It is clear that for every $\eta\in V$,
$\eta^t_\varepsilon\rightarrow \eta^t$ strongly in $V$ as
$\varepsilon\searrow0$.

\begin{theorem}
\label{tw.sob14} Assume that $f$ is a measurable function such
that \mbox{\rm(\ref{eq6.2})} is satisfied.
Then there exists a unique solution of {\rm (\ref{eq7.1})}.
\end{theorem}
\begin{proof}
Uniqueness. Let $u_1,u_2\in \VV$ be solutions of (\ref{eq7.1}). Then, by (\ref{eq7.2}),
\begin{equation}
\label{eq7.3}
\langle\partial_t \eta,u\rangle+\BB(u,\eta)=(f(\cdot,u_1)-f(\cdot,u_2),\eta)_{\HH},\quad \eta\in\WW_0,
\end{equation}
where $u=u_1-u_2$. We see that $u=G(f(\cdot,u_1)-f(\cdot,u_2))$. In particular, $u\in \WW_T$.
Taking $\eta=u^t_\varepsilon$ in (\ref{eq7.3}) and letting $\varepsilon\searrow 0$ we get
\begin{align*}
\frac12\|u(t)\|^2_H+\int_t^TB^{(s)}(u(s),u(s))\,ds
&=\int_t^T(f(s,u_1(s))-f(s,u_2(s)),u(s))_H\,ds\\&\le
\lambda^+\int_t^T\|u(s)\|^2_{H}\,ds.
\end{align*}
Applying Gronwall's lemma shows that  $u=0$.

Existence. Without lost of generality we may and will assume that $\lambda\le 0$. Let $g_1,g_2$ be as in (\ref{eq3.15}). Define $g^\alpha\in\VV'$ by
\[
\langle\langle g^\alpha,\eta\rangle\rangle=\langle\langle g,\alpha
\hat G_\alpha\eta\rangle\rangle,\quad \eta\in \VV,
\]
and to simplify notation, let  $g^{\alpha}_2$ stand for  $\alpha \hat G_{\alpha}\eta$. Then
\begin{equation}
\label{eq7.4}
\langle\langle g^\alpha,\eta\rangle\rangle=\langle g_1,\alpha \hat G_\alpha\eta\rangle-\langle\partial_t\eta,g^\alpha_2\rangle,\quad \eta\in\VV.
\end{equation}
By \cite[Theorem 6.2]{Lions}, there exists $u^\alpha\in\WW'$ such that
\[
\langle\partial_t\eta,u^\alpha\rangle+\BB(u^\alpha,\eta)=(\varphi,\eta(T,\cdot))_{H}
+(f(\cdot,u^\alpha),\eta)_{\HH}+\langle\langle g^\alpha,\eta\rangle\rangle,\quad \eta\in\VV.
\]
By this and (\ref{eq7.4}), for every $\eta\in\VV$ we have
\begin{align}
\label{eq7.6}
&\langle u^\alpha+g^\alpha_2,\partial_t\eta\rangle+\BB(u^\alpha+g^\alpha_2,\eta)\nonumber\\
&\quad=(\varphi,\eta(T,\cdot))_{H}+(f(\cdot,u^\alpha),\eta)_{\HH}+\langle g_1,\alpha \hat G_\alpha \eta\rangle+\BB(g^\alpha_2,\eta).
\end{align}
Taking $\eta=(u^\alpha+g^\alpha)^t_\varepsilon$  as a test function in (\ref{eq7.6}) and letting $\varepsilon\searrow 0$  we obtain
\begin{align}
\label{eq7.7}
&\frac12\|(u^\alpha+g^\alpha_2)(t)\|^2_H+\int_t^T B^{(s)}((u^\alpha+g^\alpha_2)(s),(u^\alpha+g^\alpha_2)(s))\,ds\nonumber \\
&\quad
=\frac12\|\varphi\|^2_H+\int_t^T(f(\cdot,u^\alpha(s)),(u^\alpha+g^\alpha_2)(s))_H\,ds+\langle g_1,\alpha \hat G_\alpha (u^\alpha+g^\alpha_2)^t\rangle \nonumber \\
&\qquad+\int_t^T B^{(s)}(g^\alpha_2(s),(u^\alpha+g^\alpha_2)(s))\,ds.
\end{align}
By \cite[Proposition I.3.7]{St2},
\[
\langle g_1,\alpha \hat G_\alpha (u^\alpha+g^\alpha_2)^t\rangle
\le C\|g_1\|_{\VV'}\|u^\alpha+g^\alpha_2\|_{\VV}.
\]
By (\ref{eq6.2}),
\begin{align*}
&\int_t^T(f(\cdot,u^\alpha(s)),(u^\alpha+g^\alpha_2)(s))_H\,ds\\
&\qquad\le M\int_t^T\|(u^\alpha+g^\alpha_2)(s)\|^2_H\,ds
+M\|g^\alpha_2\|_{\HH}\|u^\alpha+g^\alpha_2\|_\HH
+\|\varrho\|_\HH\|u^\alpha+g^\alpha_2\|_\HH.
\end{align*}
Combining the above inequalities with (\ref{eq7.7}) we get
\begin{align*}
&\frac12\|(u^\alpha+g^\alpha_2)(t)\|^2_H
+\int_t^T B^{(s)}((u^\alpha+g^\alpha_2)(s),(u^\alpha+g^\alpha_2)(s))\,ds
\nonumber\\
&\quad \le
\frac12\|\varphi\|^2_H+M\int_t^T\|(u^\alpha+g^\alpha_2)(s)\|^2_H\,ds
+M\|g^\alpha_2\|_{\HH}\|u^\alpha+g^\alpha_2\|_\HH\nonumber\\
&\qquad+\|\varrho\|_\HH\|u^\alpha+g^\alpha_2\|_\HH+C\|g_1\|_{\VV'}\|u^\alpha
+g^\alpha_2\|_{\VV}+\|g^\alpha_2\|_{\VV}\|u^\alpha+g^\alpha_2\|_\VV.
\end{align*}
Applying Gronwall's lemma yields
\begin{align*}
&\|(u^\alpha+g^\alpha_2)(t)\|^2_H+\int_0^T B^{(s)}((u^\alpha+g^\alpha_2)(s),(u^\alpha+g^\alpha_2)(s))\,ds\nonumber\\
&\quad\le C(M,T)\Big(\|\varphi\|^2_H
+\|g^\alpha_2\|_{\HH}\|u^\alpha+g^\alpha_2\|_\HH
+\|\varrho\|_\HH\|u^\alpha+g^\alpha_2\|_\HH\nonumber\\
&\qquad+\|g_1\|_{\VV'}\|u^\alpha
+g^\alpha_2\|_{\VV}+\|g^\alpha_2\|_{\VV}\|u^\alpha+g^\alpha_2\|_\VV\Big).
\end{align*}
Integrating the above inequality with respect to $t$ on $[0,T]$ we get
\begin{align*}
\nonumber\|u^\alpha+g^\alpha_2\|^2_\VV&\le \tilde C(M,T)\Big(\|\varphi\|^2_H+\|g^\alpha_2\|_{\HH}\|u^\alpha
+g^\alpha_2\|_\HH+\|\varrho\|_\HH\|u^\alpha+g^\alpha_2\|_\HH\\
&\quad+\|g_1\|_{\VV'}\|u^\alpha+g^\alpha_2\|_{\VV}
+\|g^\alpha_2\|_{\VV}\|u^\alpha+g^\alpha_2\|_\VV\Big).
\end{align*}
By  Young's inequality,
\[
\|u^\alpha+g^\alpha_2\|^2_\VV
\le \hat C(M,T)\Big(\|\varphi\|^2_H+\|g^\alpha_2\|^2_{\HH}+\|\varrho\|^2_\HH
+\|g_1\|^2_{\VV'}+\|g^\alpha_2\|^2_{\VV}\Big).
\]
By \cite[Proposition I.3.7]{St2} and (\ref{eq3.15}),
\begin{equation}
\label{eq7.12}
\|u^\alpha+g^\alpha_2\|^2_\VV
\le  C'(M,T)\Big(\|\varphi\|^2_H+\|\varrho\|^2_\HH+\|g\|^2_{\WW'_0}\Big).
\end{equation}
Therefore, up to a subsequence, $\{u^\alpha+g^\alpha_2\}$  is
weakly convergent in $\VV$ to some $v\in\VV$. By \cite[Proposition
I.3.7]{St2}, $\{g^\alpha_2\}$  strongly converges in $\VV$ to
$g_2$, so
\begin{equation}
\label{eq6.9}
u^\alpha\rightarrow u\quad\mbox{weakly in  }\VV
\end{equation}
as $\alpha\rightarrow\infty$, where
$u=v-g_2$. Observe that by (\ref{eq7.6}) and (\ref{eq7.12}), the
sequence $\{u^\alpha+g^\alpha_2\}$ is bounded in $\WW$. In
particular,  we have  $u^\alpha+g^\alpha_2, u+g_2\in\WW$.
Taking $\eta^t_\varepsilon$ as a test function  in (\ref{eq7.6})
and letting $\varepsilon\searrow 0$ we conclude that
$(u^\alpha+g^\alpha_2)(T)=\varphi$. Thus (\ref{eq7.6}) may
be rewritten as
\begin{align}
\label{eq7.13} -\langle\partial_t(u^\alpha+g^\alpha_2), \eta \rangle
+\BB(u^\alpha,\eta)=(f(\cdot,u^\alpha),\eta)_{\HH}+\langle
g_1,\alpha \hat G_\alpha \eta\rangle,\quad \eta\in \VV.
\end{align}
Let $\langle\cdot,\cdot\rangle_{V',V}$ denote the duality between  $V'$ and $V$.
Taking $(u^\alpha+g^\alpha_2-\eta)^t$ in place of $\eta$  in
(\ref{eq7.13}) we get
\begin{align}
\label{eq7.14}
\nonumber&-\int_t^T\langle\partial_s
(u^\alpha+g^\alpha_2)(s), (u^\alpha+g^\alpha_2-\eta)(s) \rangle_{V',V}\,ds
+\int_t^TB^{(s)}(u^\alpha(s),(u^\alpha+g^\alpha_2-\eta)(s))\,ds\\
&\qquad=\int_t^T(f(\cdot,u^\alpha)(s),
(u^\alpha+g^\alpha_2-\eta)(s))_{H}\,ds+\langle g_1,\alpha \hat
G_\alpha (u^\alpha+g^\alpha_2-\eta)^t\rangle.
\end{align}
By \cite[Proposition I.3.7]{St2},
\begin{equation}
\label{eq7.15} \langle g_1,\alpha \hat G_{\alpha}
(u^\alpha+g^\alpha_2-\eta)^t\rangle\rightarrow \langle g_1,
(u+g_2-\eta)^t\rangle
\end{equation}
as $\alpha\rightarrow\infty$. Thanks to the monotonicity of $L_t$ and $f$, we get
by a pseudomonotonicity argument that
\begin{align}
\label{eq7.16}
&\int_t^TB^{(s)}(u(s),(u+g_2-\eta)(s))\,ds
-\int_t^T(f(\cdot,u)(s),(u+g_2-\eta)(s))_{H}\,ds\nonumber\\
&\qquad\le  \liminf_{\alpha\rightarrow \infty }
\Big(\int_t^TB^{(s)}(u^\alpha(s),(u^\alpha+g^\alpha_2-\eta)(s))\,ds
\nonumber\\
&\qquad\qquad\qquad-\int_t^T(f(\cdot,u^\alpha)(s),
(u^\alpha+g^\alpha_2-\eta)(s))_H\,ds\Big).
\end{align}
Now observe that
\begin{align}
\label{eq7.17}
&-\int_t^T\langle\partial_s
(u^\alpha+g^\alpha_2)(s), (u^\alpha+g^\alpha_2-\eta)(s) \rangle_{V',V}\,ds
\nonumber \\
&\qquad=-\frac12\|\varphi-\eta(T)\|^2_H
+\frac12\|(u^\alpha+g^\alpha_2-\eta)(t)\|^2_H\nonumber\\
&\qquad\quad-\int_t^T\langle\partial_s \eta(s),
(u^\alpha+g^\alpha_2-\eta)(s)\rangle_{V',V}\,ds.
\end{align}
Since $C([0,T];H)\subset \WW$ and $\{u^\alpha+g^\alpha_2\}$ is
bounded in $\WW$,  we  easily get $(u+g_2)(T)=\varphi$. This when
combined with (\ref{eq7.17}) and the proved convergences shows
that for every bounded positive measurable function  $\psi$  on
$[0,T]$,
\begin{align*}
&\int_0^T\psi(t)\int_t^T\langle-\partial_s(u+g_2)(s),
(u+g_2-\eta)(s) \rangle_{V',V}\,ds\,dt\nonumber \\
&\qquad\le \liminf_{\alpha\rightarrow \infty}
\int_0^T\psi(t)\int_t^T\langle -\partial_s
(u^\alpha+g^\alpha_2)(s), (u^\alpha+g^\alpha_2-\eta)(s) \rangle_{V',V}\,ds\,dt.
\end{align*}
This when combined with (\ref{eq7.14})--(\ref{eq7.16}) shows that
\begin{align*}
&\int_0^T\psi(t)\int_t^T\langle-\partial_s(u+g_2)(s), \eta(s)
\rangle_{V',V}\,ds\,dt+\int_0^T\psi(t)\int_t^TB^{(s)}(u(s),\eta(s))\,ds\,dt\\
&\qquad\le\int_0^T\psi(t)\int_t^T(f(\cdot,u)(s),\eta(s))_H\,ds\,dt
+\int_0^T\psi(t)\langle g_1,\eta^t\rangle_{V',V}\,dt
\end{align*}
for $\eta\in \VV$. From this we get
\[
-\langle\partial_t(u+g_2),\eta\rangle +\BB(u,\eta)=(f(\cdot,u),\eta)_{\HH}+\langle g_1,
\eta\rangle,\quad \eta\in \VV. \] Since we know that
$(u+g_2)(T)=\varphi$, the above equation implies (\ref{eq7.2}).
\end{proof}

In the proposition below, we provide a stochastic representation of the solution of (\ref{eq7.1}) with $f=0$ and $g\in\WW'_0\cap\MM_{0,b}(E_{0,T})$. It will be needed in Section \ref{sec8}.

\begin{proposition}
\label{stw.sob15} Let $g\in\WW'_0$, $\varphi\in H$. Assume that
there  exist $\mu\in\MM_{0,b}(E_{0,T})$ such that
\[
\langle\langle g,\eta\rangle\rangle=\int_{E_{0,T}}\tilde\eta\,d\mu
\]
for every bounded $\eta\in\WW_0$. Let $u$ be a solution of
\mbox{\rm(\ref{eq7.1})} with $f=0$. Then for $m_1$-a.e. $x\in
E_{0,T}$,
\[
u(x)=E_x\Big(\varphi(X^0_{T})+\int_0^{\zeta}\,dA^\mu_t\Big).
\]
\end{proposition}
\begin{proof}
We  adopt the  notation from   the proof of Theorem
\ref{tw.sob14}. We first assume that $\varphi=0$. Let
$\mu_\alpha=\alpha \hat R_\alpha\circ\mu$. By (\ref{eq2.2}) and
the definition of a solution of (\ref{eq7.1}),
\[
\int_{E_{0,T}}\eta R\mu_\alpha\,dm_1=\int_{E_{0,T}}\hat R\eta\,d\mu_\alpha
=\int_{E_{0,T}}\alpha R_\alpha(\hat R\eta)\,d\mu= \langle\langle
g^\alpha,\hat R\eta\rangle\rangle=(u^\alpha,\eta)_{\HH}
\]
for every bounded $\eta\in\WW_0$. By (\ref{eq6.9}), for every $\eta\in\HH$ we have
\[
\lim_{\alpha\rightarrow\infty}(u^\alpha,\eta)_{\HH}=(u,\eta)_{\HH},
\]
whereas by (\ref{eq2.2}), for every $\eta\in\WW_0\cap C_b(E_{0,T})$,
\[
\lim_{\alpha\rightarrow\infty}\int_{E_{0,T}}\eta R\mu_\alpha\,dm_1=\lim_{\alpha\rightarrow\infty}\int_{E_{0,T}}R\mu\cdot\alpha R_\alpha\eta\,dm_1=\int_{E_{0,T}}\eta R\mu\,dm_1\,.
\]
Therefore $u=R\mu$ $m_1$-a.e. In the general case, we put
$w(x)=E_x\varphi(X^0_{T})$, $x\in E_{0,T}$. By \cite[Theorem
3.7]{K:JFA}, $w$  is a solution of (\ref{eq7.1}) with
$f=0,g=0$. Hence $v=u-w$ is a solution of
(\ref{eq7.1}) with $f=0$ and $\varphi=0$. By what has
been already proved, $v=R\mu$ $m_1$-a.e. Consequently, $u=w+R\mu$
$m_1$-a.e.
\end{proof}

\section{Further properties of $g_2$}
\label{sec6.5}

We know from Theorem \ref{th3.2} that each $\mu\in\MM_{0,b}(E_{0,T})$ admits decomposition of the form
\begin{equation}
\label{eq1.4narr} \mu=f\cdot m_1+g_1+\partial_tg_2
\end{equation}
with $f\in L^1(E_{0,T};m_1)$, $g_1\in\VV'$ and $g_2\in\VV$.
In this section, we prove some further regularity results for  $g_2$.

In the sequel, we denote by  $D([0,T];H)$  the set consisting of all functions $u\in\HH$ having an $m_1$-version
$\tilde u$ such that $[0,T]\ni t \mapsto\tilde u(t)\in H$ is c\`adl\`ag, i.e. right-continuous with left limits. Of course, $C([0,T];H)\subset D([0,T];H)$.

\begin{proposition}
Let $\mu\in \MM_{0,b}(E_{0,T})$ and let $(f,g_1,g_2)$ be a decomposition of $\mu$.
\begin{enumerate}[\rm(i)]
\item  $g_2\in D([0,T];H)$ and $g_2(T-)=0$.
\item For $t\in(0,T)$, let $\mu_t$ be the measure defined as $\mu_t(B)=\mu(\{t\}\times B)$, $B\in\BB(E)$. Then
\begin{equation}
\label{eq5.djp1}
\mu_t=(g_2(t-)-g_2(t))\cdot m.
\end{equation}
\end{enumerate}
\end{proposition}
\begin{proof}
By Lemma \ref{lem2.3},  there exists a unique $u\in\VV$ such that
\begin{equation}
\label{eq.l1}
\EE(u,\eta)=\int_{E_{0,T}}\tilde\eta\,d\mu,\quad\eta\in \WW_0.
\end{equation}
Since $\mu\in S_0(E_{0,T})$, we have $f\in L^2(E_{0,T};m_1)$.
Since $C_c(E_{0,T})\cap \WW_0$ is dense in $\WW_T$,
\begin{equation}
\label{eq.l2}
\EE(u-g_2,\eta)=\chi(\eta),\quad \eta\in \WW_0,
\end{equation}
where $\chi(\eta)=-\BB(g_2,\eta)+(f,\eta)_{\HH}+\langle\eta,g_1\rangle$.
From this we conclude that $u-g_2\in\WW_T\subset C([0,T];H)$. By
\cite[Proposition 3.2, Theorem 3.5]{K:JFA},  $u\in D([0,T];H)$. Hence
$g_2\in D([0,T];H)$. For $t\in [0,T)$, we set
\[
\rho^t_\varepsilon(s)=0,\, s\in
[0,t-\varepsilon],\quad\rho^t_\varepsilon(s)
=\varepsilon^{-1}(s-t+\varepsilon),\, s\in [t-\varepsilon,t] \quad
\rho^t_\varepsilon(s)=1,\, s\in [t,T].
\]
Let $\xi\in V$ and $\eta^t_\varepsilon=\rho^t_\varepsilon\xi$.
Because $u\in D([0,T];H)$ and $u-g_2\in C([0,T];H)$, replacing
$\eta$ by  $\eta^t_\varepsilon$ in (\ref{eq.l1}) and
(\ref{eq.l2}),  and then letting  $\varepsilon\downarrow 0$ and
$t\uparrow T$ shows that  $u(T-)=0$ and $(u-g_2)(T)=0$. Hence
$g_2(T-)=0$. Taking $\eta^t_\varepsilon$ as a test function in (\ref{eq5.12})
and letting $\varepsilon\downarrow 0$ we obtain
\begin{equation}
\label{eq5.1213ur}
\int_{[t,T)\times E}\xi\,d\mu=\int_{[t,T)\times E}f\xi\,dm_1+\langle \mathbf{1}_{[t,T)\times E}\,g_1,v\rangle-(\xi,g_2(t-))_H,
\end{equation}
from which we conclude that (\ref{eq5.djp1}) is satisfied.
\end{proof}

\begin{lemma}
\label{lm.wtwt5} Assume that $\varphi\in L^1(E;m)$, $f\in
L^1(E_{0,T})$,  and  let $\rho$ be a Borel function on
$E$ such that $0\le\rho\le 1$ and $\int_E\rho\,dm<\infty$. Define
\[
u(x)=E_x\Big(\varphi(X^0_{T})+\int_0^{\zeta}f(X_t)\,dt\Big),
\quad x\in E_{0,T}.
\]
Then $u\in C([0,T];L^1(E;\rho\cdot m))$.
\end{lemma}
\begin{proof}
Choose  $\{\varphi_n\}\subset H\cap L^1(E;m)$, $\{f_n\}\subset
\HH\cap L^1(E_{0,T};m_1)$ so  that
$\|\varphi_n-\varphi\|_{L^1(E;m)}+\|f_n-f\|_{L^1(E_{0,T};m_1)}\rightarrow
0$ as $n\rightarrow\infty$. By Proposition \ref{stw.sob15}, the
function
\[
u_n(x)= E_x\Big(\varphi_n(X^0_{T})
+\int_0^{\zeta}f_n(X_t)\,dt\Big),\quad x\in E_{0,T},
\]
is a unique solution of the Cauchy problem
\[
-\partial_t u_n-L_tu_n=f_n,\qquad u_n(T)=\varphi_n.
\]
In particular,  $u_n\in\WW\subset C([0,T];H)\subset
C([0,T];L^1(E;\rho\cdot m))$. Let $(T_{s,t})_{t>s}$ (resp. $(\hat T_{s,t})_{t>s}$) denote the semigroup determined by the form $B^{(s)}$ (resp. dual form $\hat B^{(s)}$). We have
%\begin{align*}
%\int_E E_{s,x^0}\int_0^{\zeta_\upsilon}|(f_n-f)(X_t)|\,dt\,m(dx^0)&\le
%\int_s^T\int_{E_{0,T}}|T_{s,t}(f_n-f)(t,x^0)|\,m(dx^0)\,dt\\
%&= \int_s^T\int_{E_{0,T}} (|(f_n-f)(t)|,\hat T_{s,t}1)\,dt\\
%&\le \|f_n-f\|_{L^1(E_{0,T};m_1)}
%\end{align*}
\begin{align*}
\int_E E_{s,x^0}\int_0^{\zeta}|(f_n-f)(X_t)|\,dt\,m(dx^0)&\le
\int_s^T (|T_{s,t}(f_n-f)(t)|,1)_H\,dt\\
&= \int_s^T (|(f_n-f)(t)|,\hat T_{s,t}1)_H\,dt\\
&\le \|f_n-f\|_{L^1(E_{0,T};m_1)}
\end{align*}
and
\begin{align*}
\int_{E}E_{s,x^0}|(\varphi_n-\varphi)(X^0_T)|\,m(dx^0)
&=(T_{s,T}|\varphi-\varphi_n|,1)_H\\
&=(|\varphi-\varphi_n|,\hat T_{s,T}1)_H
\le \|\varphi-\varphi_n\|_{L^1(E;m)}.
\end{align*}
Hence
\[
\sup_{t \in [0,T]}\|u_n(t)-u(t)\|_{L^1(E;\rho\cdot m)} \le
\|\varphi-\varphi_n\|_{L^1(E;m)}+ \|f_n-f\|_{L^1(E_{0,T};m_1)}.
\]
Since $u_n \in C([0,T];L^1(E;\rho\cdot m))$, $n\ge1$, this proves the lemma.
\end{proof}

The following corollary extends \cite[Lemma 2.29]{DPP}.

\begin{corollary}
Let $\mu\in\MM_{0,b}(E_{0,T})$ and  $(f,g_1,g_2)$, $(\bar f,\bar g_1,\bar
g_2)$ be two decompositions of $\mu$.  Then $g_2-\bar g_2\in
C([0,T];L^1(E;\rho\cdot m))$ for any function $\rho$ on $E$ such that
$0\le\rho\le 1$ and $\int_E\rho\,dm=1$.
\end{corollary}
\begin{proof}
For every bounded $\eta\in\WW_0$, we have
\[
\langle\partial_t\eta, g_2-\bar g_2\rangle=-\langle g_1-\bar
g_1,\eta\rangle-\int_{E_{0,T}}(f-\bar f) \eta\,dm_1.
\]
Hence,  for every  bounded $\eta\in\WW_0$,
\[
\langle\partial_t\eta, g_2-\bar g_2\rangle +\BB(g_2-\bar
g_2,\eta)=-\langle g_1-\bar g_1,\eta\rangle+\langle
\chi,\eta\rangle-\int_{E_{0,T}}(f-\bar f) \eta\,dm_1,
\]
where $\chi$ is an element of $\VV'$ such that $\langle \chi,\eta\rangle=\BB(g_2-\bar g_2,\eta)$, $\eta\in \VV$. Let $u=g_2-\bar g_2$ and  $v\in\WW_T$ be a solution to the Cauchy problem
\[
-\partial_t v-L_t v=-(g_1-\bar g_1)+\chi,\qquad v(T)=0.
\]
Then
\[
\langle\partial_t\eta,u-v\rangle+\BB(u-v,\eta)=-\int_{E_{0,T}}(f-\bar f) \eta\,dm_1
\]
for all bounded $\eta\in\WW_0$. By Proposition \ref{stw.sob15}, for $m_1$-a.e. $x\in E_{0,T}$ we have
\[
(u-v)(x)=E_x\int_0^{\zeta}(\bar f -f)(X_t)\,dt.
\]
By Lemma \ref{lm.wtwt5}, $u-v\in C([0,T];L^1(E;\rho\cdot m))$. Since $v\in\WW_T\subset C([0,T];H)\subset C([0,T];L^1(E;\rho\cdot m))$,  we get the desired result.
\end{proof}

The following proposition is a counterpart to \cite[Theorem 1.1]{PPP2}.

\begin{proposition}
Let $\mu\in\MM_{0,b}(E_{0,T})$. Then for every  $\varepsilon>0$
there exists a measure $\mu_\varepsilon \in
S_0(E_{0,T})-S_0(E_{0,T})$ such that
$\|\mu_\varepsilon-\mu\|_{TV}\le\varepsilon$ and
$\mu_{\varepsilon}$ admits decomposition of the form
\mbox{\rm(\ref{eq5.12})} with $f=0$ and  $g_2\in\VV\cap L^{\infty}(E_{0,T};m_1)$.
\end{proposition}
\begin{proof}
Without loss of generality we can assume that $\mu$ is positive.
Let $\{F^1_n\}$ be a nest such that $\mathbf{1}_{F^1_n}\cdot\mu\in
S_0(E_{0,T})$. Let $u=R\mu$, $F^2_n=\{u\le n\}$ and
$F_n=F^1_n\cap F^2_n$, $\mu_n=\mathbf{1}_{F_n}\cdot\mu$ and
$u_n=R\mu_n$. By It\^o's formula (see also (\ref{eq.deu})),
\[
|u_n(x)|^2\le
2E_x\int_0^{\zeta}u_n(X_{t-})\mathbf{1}_{F_n}(X_{t-})\,dA^\mu_t
\le 2n
E_x\int_0^{\zeta}\mathbf{1}_{F_n}(X_{t-})\,dA^\mu_t=2nu_n(x).
\]
Hence $u_n(x)\le 2n$ for q.e. $x\in E_{0,T}$. Consequently, $
u_n\in L^{\infty}(E_{0,T};m_1)$. By Proposition
\ref{stw.sob15}, $u_n$ is a solution to the  Cauchy problem
\[
-\partial_t u_n-L_tu_n=\mu_n,\qquad u_n(T)=0.
\]
In other words, for every bounded $v\in\WW_0$,
\[
\int_{E_{0,T}}\tilde v\,d\mu_n=\langle\chi_n,v\rangle
+\langle\partial_tv,u_n\rangle,
\]
where $\chi_n$ is an element of $\VV'$ such that
$\langle\chi_n,\eta\rangle=\BB(u_n,v)$, $v\in\VV$. Thus $\mu_n=\chi_n+\partial_tu_n$, i.e. $\mu_n$ admits the  decomposition \mbox{\rm(\ref{eq5.12})} with $f=0$ and $g_2:=u_n\in\VV\cap L^{\infty}(E_{0,T};m_1)$. Furthermore,
$\|\mu_n-\mu\|_{TV}=\mu(E_{0,T}\setminus F_n)\rightarrow0$ as $n\rightarrow\infty$, which completes the proof.
\end{proof}

It is known (see \cite[Example 3.1]{PPP2}) that not every bounded smooth measure can be written in the form (\ref{eq1.4narr}) with bounded $g_2$.
We are going to show that $g_2$ is always quasi-bounded
with respect to the capacity $c_2$.
To this end, we first recall  the definition of a parabolic potential (see, e.g., \cite{Pierre1}).

\begin{definition}
A measurable function $u\in\VV\cap L^\infty(0, T; H)$
is called a parabolic potential if for every positive
$v\in\WW_0$,
\begin{equation}
\label{eq8.3dm}
\langle\partial_tv,u\rangle+ \BB(u, v)\ge 0.
\end{equation}
\end{definition}

The set of parabolic potentials will be denoted by $\PP^2$. By \cite[Proposition I.1]{Pierre2}, for every $u\in\PP^2$ there exists a unique Borel measure $\mu_u$ such that for every $\eta\in \WW_0\cap C_b(E_{0,T})$,
\[
\langle \partial\eta_t, u\rangle+\BB(u,\eta)=\int_{E_{0,T}}\eta\,d\mu_u.
\]
Let $\mu\in S_0$ be positive. Then, by \cite[Proposition 3.1]{K:JEE},  $\mu\in\PP^2$ and  $\mu_{R\mu}=\mu$. Following \cite{Pierre3}, for $u\in\PP^2$ we set
\[
\|u\|^2_\Lambda=\mbox{ess\,sup}_{t\in (0,T)}\|u(t)\|^2_H+\|u\|^2_\VV.
\]
For an open set  open $U\subset E_{0,T}$, we define
\[
c_1(U)=\inf\{\|\eta\|_\Lambda:\eta\in\PP^2,\, \eta\ge \mathbf{1}_U\,\, m_1\mbox{-a.e.}\}.
\]
By \cite[Theorem 1]{Pierre3},  there exists $\alpha>0$ such that
\begin{equation}
\label{eq8.1dm}
c_0\le \alpha c_1.
\end{equation}

For $k\ge0$, we set
\[
T_k(s)=\max\{-k,\min\{k,s\}\},\quad s\in\BR.
\]
\begin{lemma}
\label{lm8.1dm}
Let $\mu\in\MM_{0,b}^+$. Then $T_k(R\mu)\in\PP^2$, $k\ge0$, and
there exists $c>0$ such that
\[
\|T_k(R\mu)\|^2_\Lambda\le ck\|\mu\|_{TV}.
\]
\end{lemma}
\begin{proof}
By \cite[Theorem 3.12]{K:JFA}, $T_k(R\mu)\in\VV$. Let $\{F_n\}$ be a Cap$_\psi$-nest of compact sets
such that $\mu_n:=\mathbf{1}_{F_n}\cdot\mu\in S_0$. Since $\{F_n\}$ is a Cap$_\psi$-nest, $R\mu_n\nearrow R\mu$ Cap$_\psi$-q.e., hence  $m_1$-a.e.  Since $R\mu_n\in \PP^2$, by \cite[Corollary I.1]{Pierre2}, $T_k(R\mu_n)\in\PP^2$.
Hence, by \cite[Lemma III-1]{Pierre1}, there exists $c>0$ such that
\[
\|T_k(R\mu_n)\|^2_\Lambda\le ck\|\mu_{T_k(R\mu_n)}\|_{TV}.
\]
By \cite[Lemma II-5]{Pierre1},
\begin{equation}
\label{eq8.2dm}
\|T_k(R\mu_n)\|^2_\Lambda\le ck\|\mu_{T_k(R\mu_n)}\|_{TV}\le ck\|\mu_{R\mu_n}\|_{TV}=ck\|\mu_n\|_{TV}\le ck\|\mu\|_{TV}.
\end{equation}
By \cite[Theorem 3.12]{K:JFA}, $\sup_{n\ge 1}\|T_k(R\mu_n)\|_\VV<\infty$.
Since $T_k(R\mu_n)\rightarrow T_k(R\mu)$ $m_1$-a.e. as $n\rightarrow\infty$, it follows that,
up to a subsequence, $T_k(R\mu_n)\rightarrow T_k(R\mu)$ weakly in $\VV$. It is clear that $T_k(R\mu)$ satisfies (\ref{eq8.3dm}), which when combined with (\ref{eq8.2dm}) implies that $T_k(R\mu)\in\PP^2$ and
the desired inequality holds true.
\end{proof}

\begin{proposition}
Assume that $\mu\in\MM_{b}(E_{0,T})$ is of the form \mbox{\rm(\ref{eq1.4})}. Then
$g_2$ has an $m_1$-version $\tilde g_2$ which is $c_0$-quasi-bounded.
\end{proposition}
\begin{proof}
Set $\nu=\mu-f\cdot m_1$ and
\[
v(x)=-E_{x}\int^{\zeta}_0dA^{\nu}_t,
\]
where $A^{\nu}$ is a natural AF of $\BX$ in the Revuz correspondence with
$\nu$. By Proposition \ref{stw.sob15}, $v$ is a solution to (\ref{eq7.1}) with $f=0$, $\varphi=0$ and $g$  replaced
by $-\nu$. Observe that $w=v-g_2$ is a
solution to the  Cauchy problem
\[
-\partial_t w-L_tw=-g_1-\chi,\qquad w(T)=0,
\]
where $\chi$ is an element of $\VV'$ such that $\langle\chi,\eta\rangle=\BB(g_2,\eta)$, $\eta\in\VV$.
Since
$g_1+\chi\in\VV'$, we have $w\in\WW_T$. Therefore
\begin{equation}
\label{eq6.5.vir1}
g_2=w-E_\cdot A^\nu_\zeta\quad m_1\mbox{-a.e.}
\end{equation}
for some $w\in\WW_T$.
By \cite[Theorem III.1]{Pierre2}, $w$ has an $m_1$-version $\tilde w$ which is  $c_0$-quasi-continuous, so $c_0$-quasi-bounded. Therefore we only need to show that  $v_+(x):= E_{x}A^{\nu^+}_\zeta$ and $v_-(x):= E_{x}A^{\nu^-}_\zeta$ have $m_1$-versions which are $c_0$-quasi-bounded. We will show this for $v_+$. The proof for $v_{-}$ is analogous. Let $v_+^k=T_k(v_+)$. By Lemma \ref{lm8.1dm}, $v_+^k\in\PP^2$, so by \cite[Corollary III.3]{Pierre2}, there exists an $m_1$-version $\tilde v_+^k$
of $v_+^k$ which is $c_0$-quasi-l.s.c. Hence $\tilde v_+:=\sup_{k\ge 1}\tilde v_+^k $ is also $c_0$-quasi-l.s.c.,  and of course, it is an $m_1$-version of $v_+$.
By Lemma \ref{lm8.1dm} and (\ref{eq8.1dm}),
\begin{equation}
\label{eq8.4dm}
c_0(\tilde v_+>n)=c_0(\tilde v^{n+1}_+>n)\le \alpha n^{-2}\|(v^{n+1}_+)\|^2_\Lambda\le2c\alpha n^{-1}\|\nu^+\|_{TV}.
\end{equation}
Let $\{F^1_n\}$ be a $c_0$-nest such that $(\tilde v_+)_{|F_n}$ is l.s.c. for each $n\ge 1$. Set $F_n=F^1_n\cap \{\tilde v_+\le n\}$.
Then  $F_n$ is closed. Moreover,
\begin{align*}
c_0(E_{0,T}\setminus F_n)&\le c_0(\tilde v_+>n)+c_0(E_{0,T}\setminus F^1_n)\rightarrow 0
\end{align*}
as $n\rightarrow\infty$, which proves that $\tilde v_+$ is $c_0$-quasi-bounded.
%Let $(\phi,\gamma_1,\gamma_2)$ be
%a decomposition for $\nu^+$ of Theorem \ref{}. By the reasoning preceding ...
%\[
%R\nu^+=u+R\psi,
%\]
%where $u\in\WW_T$ is a unique solution to the Cauchy problem
%\begin{equation}
%\label{eq.sob10}
%(\partial_t+L_t)u=\gamma_1,\qquad u(T)=0,
%\end{equation}
%Repeating the reasoning of the proof of Theorem \ref{th8.1} but for $\nu^+$
\end{proof}

\section{Smoothness of measures in $\WW'_0$ }
\label{sec7}

We  already know that each bounded smooth measures admit decomposition of the form (\ref{eq1.4}). The problem whether a bounded measure admitting decomposition (\ref{eq1.4}) is smooth is more delicate. In this section, we give  positive answer to this question.
% in the case where $\EE$ satisfies the so-called continuity condition (also called Meyer's condition (L); see \cite[p. 246]{DM2}).
%In the general case, we are able to give positive answer under the additional assumption that $g_2\in L^{\infty}(E_{0,T};m_1)$.

In the proof of our result, we will make use of Fukushima's decomposition, which we now recall. For an AF $A$ of $\BX$ its energy is defined by
\[
e(A)=\frac12\lim_{\alpha\rightarrow\infty}\alpha^2E_{m_1}\int^{\infty}_0e^{-\alpha t}|A_t|^2\,dt
\]
whenever the limit exists in $[0,\infty]$.
By \cite[Proposition IV.1.8]{St2}, each $u\in\WW$ has a quasi-continuous $m_1$-modification, which we will denote by $\tilde u$. By \cite[Theorem 4.5]{T1},  for any $w\in\WW$ there exists a unique  martingale AF of $\BX$ of finite energy  $M^{[u]}$ and a unique continuous AF of $\BX$ of zero energy $N^{[u]}$ such that
\begin{equation}
\label{eq2.6} \tilde u(X_t)-\tilde
u(X_0)=M^{[u]}_t+N^{[u]}_t,\quad t\ge0,\quad P_x\mbox{-a.s.}
\end{equation}
for q.e. $x\in E_{0,T}$. The decomposition (\ref{eq2.6}) is called  Fukushima's decomposition.

\begin{proposition}
\label{prop7.1}
Assume that $u\in \WW_T$. Then
\[
e(M^{[u]})\le \BB(u,u).
\]
\end{proposition}
\begin{proof}
Let $A^{[u]}_t=\tilde u(X_t)-\tilde u(X_0)$, $t\ge0$. We have
\begin{align*}
\alpha^2E_{m_1}\int^{\infty}_0e^{-\alpha t}|A^{[u]}_t|^2\,dt
&=\alpha^2\int_{E_{0,T}}(R_{\alpha}u-2u R_{\alpha}u+\alpha^{-1}u)\,dm_1\\
&=2\alpha(u-\alpha R_{\alpha}u,u)_{\HH}-\alpha(u^2,1-\alpha \hat R_{\alpha}1)_{\HH}\\
&=2\EE(\alpha G_{\alpha}u, u)-\alpha(u^2,1-\alpha \hat R_{\alpha}1)_{\HH}.
\end{align*}
By (\ref{eq2.30}),
$\alpha\hat R_\alpha 1(x)=1-\hat E_x e^{-\alpha\zeta}$, $x\in E_{0,T}$.
Hence
\[
\alpha(u^2,1-\alpha \hat R_{\alpha}1)_{\HH}
=\int_{E_{0,T}}\alpha|u|^2\hat E_x e^{-\alpha\zeta}\,m_1(dx)
\ge \int_0^T\alpha |u(s)|_H^2\,e^{-\alpha s}\,ds,
\]
which converges to $\|u(0)\|^2_H$ as $\alpha\rightarrow\infty$. By the above calculations,
\begin{align*}
e(A^{[u]})&=\limsup_{\alpha\rightarrow \infty}\Big(\EE(\alpha G_{\alpha}u,u)-\frac12\alpha(u^2,1-\alpha \hat R_{\alpha}1)_{\HH}\Big),
\end{align*}
so by \cite[Proposition 2.7]{St1},
\[
e(A^{[u]})\le\EE(u,u)-\frac12\|u(0)\|^2_H=\BB(u,u).
\]
On the other hand, by  (\ref{eq2.6}), $e(A^{[u]})=e(M^{[u]})$, which proves the proposition.
\end{proof}

\begin{lemma}
\label{lm8.c2}
Assume that $A, A^n$, $n\ge 1$, are additive functionals of $\mathbb X$.
\begin{enumerate}[\rm(i)]
\item If $E_{\eta\cdot m_1}\sup_{t\ge 0}|A_t|^p<\infty$ for some $p\ge 1$ and strictly positive $\eta\in\BB(E_{0,T})$, then $E_x\sup_{t\ge 0}|A_t|^p<\infty$ for  q.e. $x\in E_{0,T}$.
\item If $E_{\eta\cdot m_1}\sup_{t\ge 0}|A^n_t-A^m_t|^p\rightarrow 0$ as $n,m\rightarrow \infty$ for some $p\ge 0$ and strictly positive $\eta\in\BB(E_{0,T})$, then there exists
a subsequence (still denoted by $n$) such that $E_{x}\sup_{t\ge 0} |A^n_t-A^m_t|^p\rightarrow 0$ as $n,m\rightarrow \infty$ for  q.e. $x\in E_{0,T}$.
\end{enumerate}
\end{lemma}
\begin{proof}
We will prove (ii). The proof of (i) is analogous. We may assume that
\begin{equation}
\label{eq8.16.1}
 E_{\eta\cdot m_1}\sup_{t\ge 0}|A^n_t-A^{n+1}_t|^p\le 2^{-n},\quad n\ge 1,
\end{equation}
and $\int_{E_{0,T}}\eta\,dm_1=1$. Let  $\tilde m_1=\eta\cdot m_1$ and
\[
B=\liminf_{n\rightarrow\infty}\{\sup_{t\ge 0}|A^n_t-A^{n+1}_t|^p>2^{-2n}\},
\qquad N=\{x\in E_{0,T}:P_x(B)>0\}.
\]
It is clear that $N$ is a nearly Borel set. We will show that $\mbox{Cap}_\psi(N)=0$.
Since Cap$_\psi$ is a Choquet capacity, we may and will assume that $N$ is compact.
We have
\[
P_{\tilde m_1}(\sigma_N<\infty)\le P_{\tilde m_1}(X_{\sigma_N}\in N)
=P_{\tilde m_1}(E_{X_{\sigma_N}}\fch_B>0).
\]
By the strong Markov property,
\[
P_{\tilde m_1}(E_{X_{\sigma_N}}\fch_B>0)
=P_{\tilde m_1}(E_x(\fch_{B}\circ\theta_{\sigma_N}|\FF_{\sigma_N})>0).
\]
Since $A$ is additive,
\begin{align*}
B\circ\theta_{\sigma_N}&=\liminf_{n\rightarrow\infty}\{\sup_{t\ge 0} |A^n_t\circ\theta_{\sigma_N}-A^{n+1}_t\circ\theta_{\sigma_{N}}|^p>2^{-n}\}\\
&\subset\liminf_{n\rightarrow\infty}\{\sup_{t\ge 0}|A^n_t-A^{n+1}_t|^p>2^{-n+p}\}
=:B'.
\end{align*}
Hence
\[
P_{\tilde m_1}(\sigma_N<\infty)\le \int_{E_{0,T}}P_x(E_x(\fch_{B'}|\FF_{\sigma_N})>0)
\tilde m_1(dx).
\]
But by (\ref{eq8.16.1}) and the Borel-Cantelli lemma, $P_{\tilde m_1}(B')=0$,
which implies that $P_x(B')=0$ for $m_1$-a.e. $x\in E_{0,T}$. Hence,  for $m_1$-a.e. $x\in E_{0,T}$, $P_x(E_x(\fch_{B'}|\FF_{\sigma_N})>0)=0$, so
$P_{\tilde m_1}(\sigma_N<\infty)=0$. Eqivalently, Cap$_\psi(N)=0$, as claimed.
\end{proof}

\begin{lemma}
\label{lm8.c1}
The set
$C=\{\sum_{i=1}^n\xi_iv_i:\xi_i\in C_c^\infty((0,T)),\, v_i\in V\cap C_c(E) \}$
is  dense  in $\WW^2_0$.
\end{lemma}
\begin{proof}
Follows from \cite[Lemma 1.1]{O3}.
\end{proof}

\begin{theorem}
\label{tw.sob3}
Let $\mu\in\MM_b(E_{0,T})$. If there exist $f\in L^1(E_{0,T};m_1)$, $g_1\in\VV'$,  $g_2\in\VV$ such that \mbox{\rm(\ref{eq5.12})} is satisfied for
all $\eta\in \WW_0\cap C_b(E_{0,T})$, then $\mu\in\MM_{0,b}(E_{0,T})$.
\end{theorem}
\begin{proof}
Since the measure $f\cdot m_1$ is smooth, without loss of generality we may assume that $f=0$. Let $\Phi$ be a
functional on $\WW_0$ defined by the right-hand side of
(\ref{eq5.12}) (with $f=0$). It is clear that $\Phi\in\WW'_0$. Let
$u\in\VV$ be a solution to (\ref{eq7.1}) with $\varphi=0$, $f=0$ and $g=\Phi$.

{\em Step 1.}  We will show that $u$ is a difference of excessive functions.
By (\ref{eq5.12}) and the definition of solution to (\ref{eq7.1}), for every $\eta\in \WW_0\cap C_b(E_{0,T})$,
\begin{equation}
\label{eq8.c4}
|\EE(u,\eta)|\le \|\mu\|_{TV}\|\eta\|_\infty.
\end{equation}
Let $\mathcal C=\{\sum_{i=1}^n\xi_iv:\xi_i\in H^1(0,T),\, \xi_i(0)=0,\, v\in V\cap C_c(E)\}$.
Then $\mathcal C\subset\WW_0\cap C_b(E_{0,T})$, so (\ref{eq8.c4}) holds for every $\eta\in \mathcal C$. By Lemma \ref{lm8.c1}, $\mathcal C$
is dense in $\WW^2_0$. Let $\eta\in \WW^2_{0}$ be bounded. Write $c=\|\eta\|_\infty$ and choose $\{\eta_n\}\subset\mathcal C$ such that
and $\eta_n\rightarrow \eta$ in $\WW^2_0$. Then  $T_c(\eta_n)\rightarrow T_c(\eta)=\eta$
in $\WW^2_0$, so  (\ref{eq8.c4}) holds for every bounded $\eta\in \WW^2_{0}$. Suppose now that $\eta\in \WW_{0}$ and $\eta$ is bounded. Then $\alpha\hat R_\alpha\eta\in \WW^2_{0}$,   $\|\alpha\hat R_\alpha\eta\|_\infty\le \|\eta\|_\infty$ and $\alpha \hat R_\alpha\eta\rightarrow \eta$ in $\WW_0$ as $\alpha\rightarrow\infty$.
Therefore  (\ref{eq8.c4}) holds for every bounded $\eta\in\WW_{0}$.
By \cite[Proposition 4.5]{BC}, there exist excessive functions $v$ and $w$  such that $u=v-w$ $m_1$-a.e. and for every $\eta\in\HH\cap L^\infty(E_{0,T};m_1)$,
\begin{equation}
\label{eq8.c9}
\int_{E_{0,T}}(v+w)|\eta|\,dm_1<\infty,
\end{equation}
\begin{equation}
\label{eq8.c10}
\frac1t\int_{E_{0,T}}|P_tv-v|\eta\,dm_1 + \frac1t\int_{E_{0,T}}|P_tw-w|\eta\,dm_1\le c\|\eta\|_\infty.
\end{equation}
Since each  excessive function finite $m_1$-a.e. is finite q.e.,  we may assume that
$v(x)+w(x)<\infty$, $x\in E_{0,T}\setminus N$, where $N$ is some $m_1$-inessential set (see \cite[Proposition 6.12]{GS}).
%Since $v,w$ are excessive functions, for $\alpha\le\bar\alpha$ we have $\alpha R_\alpha v\le \bar\alpha R_{\bar\alpha} v$ and $\alpha R_\alpha w\le \bar\alpha R_{\bar\alpha} w,\, m_1$-a.e. Since $\alpha R_\alpha v, \alpha R_\alpha w$
%are quasi-continuous  the latter inequalities hold q.e. Thus there exists an exceptional set $N$ (we may assume that $N$ is an absorbing set) such that $\alpha R_\alpha v (x), \alpha R_\alpha w(x)$ are nondecreasing with respect to
%$\alpha$ for $x\in E_{0,T}\setminus N$. Let us put
%\[
%\tilde v(x)=\sup_{\alpha>0} \alpha R_\alpha v(x),\quad \tilde w(x)=\sup_{\alpha>0}\alpha R_\alpha w(x).
%\]
%Of course $\tilde v=v,\, \tilde w=w,\, m_1$-a.e. Moreover
% $\alpha R_\alpha \tilde v(x)\le \tilde v(x),\, \alpha R_\alpha\tilde w(x)\le\tilde w(x)$,
%$x\in E_{0,T}\setminus N$, $\alpha>0$ and $\alpha R_\alpha\tilde v(x)\nearrow \tilde v(x)$, $\alpha R_\alpha\tilde w(x)\nearrow \tilde w (x)$, $x\in E_{0,T}\setminus N$ as $\alpha\rightarrow \infty$. In other words $\tilde v,\, \tilde w$
%are excessive functions with respect to the process $\mathbb X^{E_{0,T}\setminus N}$. Thus by \cite[Theorem III.5.7]{BG} $\tilde v(X),\, \tilde w(X)$ are c\`adl\`ag supermartingales under measure $P_x$ for $x\in E_{0,T}\setminus N$.

{\em Step 2.} Let
\[
\tilde u= v-w,\qquad u_n=nR_n\tilde u,\qquad \nu_n=(n\tilde u-n^2R_n \tilde u)\cdot m_1.
\]
%Let us define $\mu_n:=n\hat R_n\circ\mu$. Observe that by (\ref{eq8.1}) for every $\eta\in C_b(E_{0,T})$
%\begin{align*}
%(nu-n^2R_n u,\eta)_\HH=\EE(u,n\hat R_n\eta)=\Phi(n\hat R_n\eta)=\int_{E_{0,T}}n\hat R_n\eta\,d\mu
%=\int_{E_{0,T}}\eta\,d\mu_n.
%\end{align*}
%So, $\mu_n=(nu-n^2R_n u)\cdot m_1$.
By \cite[Proposition I.3.7]{St2}, $u_n\rightarrow u$ strongly in $\VV$.
By  Fukushima's decomposition, there exists an  $m_1$-inessential set $N$ such that for every $x\in
E_{0,T}\setminus N$,
\begin{equation}
\label{eq8.5} u_n(X_t)=u_n(X_0)-\int_0^t\,dA^{\nu_n}_r
+\int_0^t\,dM^{[u_n]}_r,\quad t\le\zeta,\quad
P_x\mbox{-a.s.}
\end{equation}
We will show the uniform convergence of $\{M^{[u_n]}\}$ as
$n\rightarrow\infty$.
Let $\langle M^{[u_n]}\rangle$ denote the sharp bracket of $M^{[u_n]}$
(see, e.g, \cite[Section A.3]{FOT}).
We have
\begin{align}
\label{eq7.5}
e(M^{[u_n]})&=\frac12\lim_{\alpha\rightarrow\infty}\alpha^2
E_{m_1}\int^{\infty}_0e^{-\alpha t} E_{m_1}\langle M^{[u_n]}\rangle_t\,dt\nonumber \\ &=\frac12\lim_{\alpha\rightarrow\infty}
\int^{\infty}_0se^{-s}\frac{\alpha}{s}E_{m_1}\langle M^{[u_n]}\rangle_{s/\alpha}\,ds.
\end{align}
Since $t\mapsto E_{m_1}\langle M^{[u_n]}\rangle_t$ is subadditive, $(1/t)E_{m_1}\langle M^{[u_n]}\rangle_t$ increases as $t$ decreases, and $\lim_{t\downarrow0}(1/t)E_{m_1}\langle M^{[u_n]}\rangle_t=\sup_{t>0}(1/t)E_m\langle M^{[u_n]}\rangle_t$. Therefore, letting $\alpha\rightarrow\infty$ in (\ref{eq7.5}), shows that
\[
e(M^{[u_n]})=
\frac12 \sup_{t>0}\frac{1}{t}E_{m_1}\langle M^{[u_n]}\rangle_t
\cdot\int^{\infty}_0se^{-s}\,ds
=\frac12 \sup_{t>0}\frac{1}{t}E_{m_1}\langle M^{[u_n]}\rangle_t.
\]
Hence
\[
E_{m_1}\langle M^{[u_n]}\rangle_{\zeta}\le 2Te(M^{[u_n]}),
\]
so applying  Doob's inequality we get
\[
E_{m_1}\sup_{t\ge 0}|M^{[u_n]}_t-M^{[u_m]}_t|^2\le 8Te(M^{[u_n-u_m]})
\]
for all $n,m\ge1$. By  Proposition \ref{prop7.1} and  Lemma \ref{lm8.c2}, there exists an  $m_1$-inessential set $N$
such that, up to a subsequence,
%\begin{equation}
%\label{eq8.3}
\[
\lim_{n,m\rightarrow\infty}E_x\sup_{t\ge 0}
|M^{[u_n]}_t-M^{[u_m]}_t|^2=0,\quad x\in E_{0,T}\setminus N.
\]
%\end{equation}

{\em Step 3. } We will show that there exists a natural AF $A$ of $\mathbb X$ of finite variation and a martingale AF $M$ of $\mathbb X$ such that
\begin{equation}
\label{eq8.6} \tilde u(X_t)=\tilde u(X_0)-\int_0^t\,dA_r
+\int_0^t\,dM_r,\quad t\le\zeta,\quad P_x\mbox{-a.s.}
\end{equation}
for $x\in E_{0,T}\setminus N$. Set
\[
M_t=\liminf_{n\rightarrow \infty}M^{[u_n]}_t,\qquad A_t=\liminf_{n\rightarrow \infty}A^{\nu_n}_t,\quad  t\ge 0.
\]
By the definition of $\tilde u$,  for every $x\in E_{0,T}$,
%\begin{equation}
%\label{eq8.c8}
\[
u_n(X_t)\rightarrow \tilde u(X_t),\quad t\ge 0,\quad P_x\mbox{-a.s.}
\]
%\end{equation}
Therefore, letting $n\rightarrow\infty$ in (\ref{eq8.5}), shows that
(\ref{eq8.6}) is satisfied. Moreover,  $M$ is a square integrable  martingale
under the measure $P_x$ for $x\in E_{0,T}\setminus N$, and
\begin{equation}
\label{eq8.c12}
E_x\sup_{t\ge 0}|M^{[u_n]}_t-M_t|^2\rightarrow 0,\qquad A_t=\lim_{n\rightarrow \infty}A^{\nu_n}_t,\quad t\ge 0,\, P_x\mbox{-a.s.}
\end{equation}
for every $x\in E_{0,T}\setminus N$.
By  \cite[Theorem III.5.7]{BG}, $v(X)$, $w(X)$ are c\`adl\`ag supermartingales under the measure $P_x$ for $x\in E_{0,T}\setminus N$. In particular, $\tilde u(X)$ and $A$ are   c\`adl\`ag processes under $P_x$ for $x\in E_{0,T}\setminus N$.
By the resolvent identity,
\begin{align}
\label{eq8.c11}
\nonumber R|\nu_n|&=nR|u-nR_nu|\le nR|v-nR_nv|+nR|w-nR_nw| \\
&=nR(v-nR_nv)+nR(w-nR_nw)=nR_nv+nR_nw\le v+w.
\end{align}
Hence
\begin{equation}
\label{eq8.c13}
\sup_{n\ge 1}E_x|A^{\nu_n}|_\zeta\le v(x)+w(x)<\infty,\quad x\in E_{0,T}\setminus N.
\end{equation}
Let $t^k_i=iT/2^k,\, i=0,\dots,2^k$. Since $A$ is c\`adl\`ag,
\[
|A|_T=\lim_{k\rightarrow \infty}\sum_{i=1}^k|A_{t^k_{i}}-A_{t^k_{i-1}}|=\sup_{k\ge 1}\sum_{i=1}^k|A_{t^k_{i}}-A_{t^k_{i-1}}|.
\]
Hence
\begin{align*}
E_x|A|_T&=E_x \lim_{k\rightarrow \infty}\sum_{i=1}^k|A_{t^k_{i}}-A_{t^k_{i-1}}|= \lim_{k\rightarrow \infty} E_x\sum_{i=1}^k|A_{t^k_{i}}-A_{t^k_{i-1}}|
\\&\le\lim_{k\rightarrow \infty}\liminf_{n\rightarrow \infty} E_x\sum_{i=1}^k|A^{\nu_n}_{t^k_{i}}-A^{\nu_n}_{t^k_{i-1}}|
\le \liminf_{n\rightarrow \infty} E_x\sup_{k\ge 1} \sum_{i=1}^k|A^{\nu_n}_{t^k_{i}}-A^{\nu_n}_{t^k_{i-1}}|\\&
= \liminf_{n\rightarrow \infty} E_x\lim_{k\rightarrow \infty}\sum_{i=1}^k|A^{\nu_n}_{t^k_{i}}-A^{\nu_n}_{t^k_{i-1}}|
= \liminf_{n\rightarrow \infty} E_x|A^{\nu_n}|_T.
\end{align*}
%Let $\mathcal C$ be a countable subset of $C_c(E_{0,T})$ such that every $\eta\in\mathcal C$
%is  positive, $\int_{E_{0,T}}\eta\,dm_1<\infty$ and $\mathcal C$ is dense in $C_b(E_{0,T})$
%with respect to the uniform convergence on compacts.
%By \cite[Theorem 3.8]{J} sequence $\{A^{\nu_n}\}$ is $S$-tight with respect to the measure $P_x$. So, by \cite[Theorem 3.5, Theorem 3.10]{J} there exists a subsequence (still denoted by $(n)$)
%such that
%\[
%A^{\nu_n}\rightarrow V^{x},
%\]
%weakly in topology $S$ (with respect to the measure $P_x$) for some
% c\`adl\`ag increasing processes $V^{1,x}, V^{2,x}$ such that $E_xV^{1,x}_\zeta+E_xV^{2,x}_\zeta<\infty$. In particular
% \[
 %A^{\nu_n}\rightarrow V^{1,x}-V^{2,x},\quad  A^{|\nu_n|}\rightarrow V^{1,x}+V^{2,x}
 %\]
 %weakly in topology $S$ (under measure $P_x$).
%By   (\ref{eq8.3}), (\ref{eq8.c8}) and \cite[Theorem 3.5]{J} $A$ and $V^{1,x}-V^{2,x}$ have the same distribution under the measure $P_x$.
From this and (\ref{eq8.c13}) we get
\begin{equation}
\label{eq8.c5}
E_x|A|_\zeta<\infty,\quad x\in E_{0,T}\setminus N.
\end{equation}
Let
\[
\Lambda=\{\omega\in\Omega:A_t(\omega)=\limsup_{n\rightarrow \infty}A^n_t(\omega),\, t\ge 0\,\, A_\cdot(\omega)\,\,\mbox{ is c\`adl\`ag of finite variation}\}.
 \]
By what  has already been proved, $P_x(\Lambda)=1$ for  $x\in E_{0,T}\setminus N$.
Observe that $\theta_t(\Lambda)\subset \Lambda$.  Moreover, for all $s,t\ge 0$ and $\omega\in\Lambda$,
\begin{equation}
\label{eq8.c1}
A_t(\theta_s(\omega))=\lim_{n\rightarrow\infty}A^{\nu_n}_t(\theta_s(\omega))=
\lim_{n\rightarrow \infty}(A^{\nu_n}_{t+s}(\omega)-A^{\nu_n}_s(\omega))=A_{t+s}(\omega)-A_s(\omega).
\end{equation}
Of course, the same relation  holds for $A^+,A^-$. Thus $A^+,A^-$ are positive natural AFs of $\mathbb X$ with the defining set $\Lambda$ and exceptional set $N$. This implies that $M$ is a martingale AF of $\mathbb X$.

{\em Step 4.}
Let $\nu^+$ (resp. $\nu^-$) be the Revuz measure associated with $A^+$ (resp. $A^-$) (see \cite[Section 8]{GS}).
To complete the proof it suffices to show that $\nu\in\MM_{0,b}(E_{0,T})$ and $\nu=\mu$.
Since $M$ is a uniformly integrable martingale,
\[
\tilde u(x)=E_xA_\zeta,\quad x\in E_{0,T}\setminus N.
\]
We have $\tilde u=u_+-u_-$\,, where
\[
u_+(x):=E_xA^+_\zeta,\qquad u_-(x):=E_xA^-_\zeta,\quad x\in E_{0,T}\setminus N.
\]
It is clear that $u_+,\, u_-$ are natural potentials (see \cite[Definition IV.4.17]{BG} and the comments following the definition).
By the construction of functions $v,w$ (see \cite[Proposition 4.2]{BC}),
\[
v(x)\le u_+(x),\qquad w(x)\le   u_-(x),\quad x\in E_{0,T}\setminus N.
\]
Therefore $v,w$ are natural potentials, and by \cite[Theorem IV.4.22]{BG}, $v=E_\cdot A^1_\zeta$\,, $w=E_\cdot A^2_\zeta$
for some positive natural AFs $A^1, A^2$ of $\mathbb X$.  By the minimality argument,
\[
v(x)= u_+(x),\qquad w(x)=  u_-(x),\quad x\in E_{0,T}\setminus N.
\]
%Moreover this decomposition is minimal, i.e. if $u=u_1-u_2$ for some natural potentials
%$u_1, u_2$, then $u_+\le u_1,\, u_-\le u_2$.
From this, (\ref{eq8.c10}), (\ref{eq8.c11}) and \cite[Theorem 9.3]{GS} it follows that for every positive $\eta\in C_b(E_{0,T})$,
\[
\frac1tE_{\eta\cdot m}|A|_t\le \frac1t\int_{E_{0,T}}|P_tv-v|\eta\,dm_1 + \frac1t\int_{E_{0,T}}|P_tw-w|\eta\,dm_1\le c\|\eta\|_\infty.
\]
By \cite[Theorem 8.7]{GS}, for every positive $\eta\in C_b(E_{0,T})$,
\[
\lim_{t\downarrow0}\frac1tE_{\eta\cdot m}|A|_t= \int_{E_{0,T}}\eta\,d|\nu|.
\]
Therefore $\int_{E_{0,T}}\eta\,d|\nu|<\infty$ for any positive $\eta\in C_b(E_{0,T})$. Consequently, $\nu\in\MM_b(E_{0,T})$.
By the very definition of  $\nu^+$ and $\nu^-$, if Cap$_\psi(B)=0$, then $|\nu|(B)=0$, so $\nu\in\MM_{0,b}(E_{0,T})$.
By \cite[Theorem 9.3]{GS}, $\tilde u= R\nu$ q.e.
By the definition of a solution to (\ref{eq7.1}) and (\ref{eq5.12}),
\[
\EE(u,\eta)=\int_{E_{0,T}}\eta\,d\mu
\]
for $\eta\in\WW_{0}\cap C_b(E_{0,T})$. On the other hand,  for every $\eta\in\WW_{0}\cap C_b(E_{0,T})$,
\[
\EE(u,\eta)=\lim_{t\downarrow0}\frac{1}{t}(u,\eta-\hat T_t\eta)=\lim_{t\downarrow0} \frac1t(u-T_tu,\eta)=\lim_{t\downarrow0}\frac{1}{t}E_{\eta\cdot m_1} \int_0^t\,dA^\nu_r=\int_{E_{0,T}}\eta\,d\nu.
\]
Hence $\mu=\nu$.
\end{proof}

\section{Decomposition of measures and additive functionals}
\label{sec8}

In this section, we study the structure of the additive
functional $A^{\mu}$ in the Revuz correspondence with
$\mu\in\MM_{0,b}(E_{0,T})$. Specifically, we want to get deeper
understanding of the nature of jumps of $A^{\mu}$ and their
relation to the decomposition $(f,g_1,g_2)$  of $\mu$ given in
Theorem \ref{th3.2}. Before proceeding,  remarks concerning two
special cases of $\mu$ are in order.

If $\mu=f\cdot m_1$, then
\[
A^{\mu}_t=\int^t_0f(X_r)\,dr,\quad t\ge0.
\]
If $\mu=g_1$, then
\begin{equation}
\label{eq8.8}
A^{\mu}_t=-N^{[u]}_t,\quad t\ge0,
\end{equation}
where $u\in\WW_T$ is a unique solution to the Cauchy problem
\begin{equation}
\label{eq.sob10}
(\partial_t+L_t)u=-g_1,\qquad u(T)=0,
\end{equation}
and $N^{[u]}$ is the continuous AF from  Fukushima's decomposition
(\ref{eq2.6}).  Indeed,
%let
%$g^\alpha_1(\eta):= \langle g_1,\alpha\hat R_\alpha\eta\rangle$. Observe that $g^\alpha_1\in\HH\subset\VV'$.
%Let $u_\alpha$ be a solution to (\ref{eq.sob10}) with $g_1$ replaced by $g^\alpha_1$.
%It is clear by \cite[Proposition 3.7]{St2} that $\langle g^\alpha_1,\eta\rangle\rightarrow \langle g_1,\eta\rangle$
%as $\alpha\rightarrow \infty$ for every $\eta\in\VV$. Hence we get easily that (up to subsequence)$u_\alpha\rightarrow u$ weakly in $\WW_T$. Let $\eta\in \WW_0\cap C_b(E_{0,T})$, then by (\ref{eq2.2}), (\ref{eq5.12})
%\[
% (u_\alpha,\eta)=\EE(u_\alpha,\hat R\eta)=
%\langle g^\alpha_1,\hat R\eta\rangle=\int_{E_{0,T}}\alpha R_\alpha(\hat R\eta)\,d\mu=(\alpha R_\alpha\eta,R\mu).
%\]
%Passing to the limit with $\alpha\rightarrow \infty$ we get by continuity of $\eta$ and weak convergence of $u_\alpha$ that
by Proposition \ref{stw.sob15}, $u=R\mu$ $m_1$-a.e. In  other
words, for q.e. $x\in E_{0,T}$ we have
\begin{equation}
\label{eq8.10}
\tilde u(x)=E_{x}A^\mu_{\zeta}.
\end{equation}
Let $M^x_t=E_{x}(A^\mu_{\zeta}|\FF_t)-\tilde u(X_0)$, $t\ge0$. Because of standard perfection procedure (see, e.g., \cite[Lemma A.3.6]{FOT}), there is a martingale AF $M$ of $\BX$ such that $M_t=M^x_t$, $t\ge0$, $P_x$-a.s. for q.e. $x\in E_{0,T}$.  From (\ref{eq8.10}) and the strong Markov property we obtain
\begin{equation}
\label{eq.deu}
\tilde u(X_t)-\tilde u(X_0)=-A^\mu_t+M_t,\quad t\ge 0
\end{equation}
(see \cite[Remark 3.3]{KR:NoD} for more details).
Since $\tilde u$ is quasi-continuous and the filtration is quasi-left continuous,  from (\ref{eq.deu}) it follows  that  $A^\mu$ is continuous. It is clear
that $M$ is a martingale AF of $\BX$. Let $u_n=nR_nu$. Then
\begin{equation}
\label{eq.deu.cc}
 u_n(X_t)-u_n(X_0)=N^{[u_n]}_t+M^{[u_n]}_t,\quad t\ge 0,
\end{equation}
and by  It\^o's formula,
\begin{equation}
\label{eq.deu.cc1}
 u_n(X_t)- u_n(X_0)=-A^{\mu_n}_t+M^n_t,\quad t\ge 0,
\end{equation}
where $\mu_n=n(u_n-u)\cdot m_1$ and $M^n_t=E_{x}(A^{\mu_n}_{\zeta}|\FF_t)-\tilde u_n(X_0)$ $P_x$-a.s. for q.e. $x\in E_{0,T}$.
By an elementary calculation (see, e.g., \cite[page 245]{FOT}),
$A^{\mu_n}$ is of zero energy. By uniqueness of Fukushima's
decomposition, $-A^{\mu_n}=N^{[u_n]}$. We know that $u_n\rightarrow \tilde u$ q.e. and $u_n\rightarrow u$ in $\VV$, so by Proposition \ref{prop7.1}, $N^{[u_n]}\rightarrow N^{[u]}$. Also, by \cite[Theorem 3.3]{K:SPA},
$A^{\mu_n}_t\rightarrow A^\mu_t,\, t\ge 0$. Thus $-A^\mu=N^{[u]}$.

We see that in both special cases  considered above the  additive
functionals corresponding to $\mu$ are continuous. This suggests
that the jumps of $A^{\mu}$ stem from the component $g_2$ of the
decomposition of $\mu$. In what follows we will show that this is indeed
true and  we will make this statement more precise.

%To make this statement more precise, we
%need some results on solutions of the linear Cauchy problem with
%data in $\WW'_0\cap\MM_{0,b}(E_{0,T})$. ??

Following \cite{K:JFA} we adopt the following definition.

\begin{definition}
We say that a Borel measurable function $u$ on $E_{0,T}$  is
quasi-c\`adl\`ag if for q.e. $x\in E_{0,T}$ the process $t\mapsto
u(X_t)$ is c\`adl\`ag on $[0,T-\tau(0))$ under the measure $P_x$.
\end{definition}

Since $A^{\mu}$ is predictable, by \cite[Chapter IV, Theorem 88B]{DM},
%(albo Lipcer, Shiryayev, Teorei martyngalow, Tw.1.6).
there  is a sequence $\{\tau_n\}$ of predictable stopping times
exhausting the jumps of $A^{\mu}$, i.e.
\[
\{\Delta A^{\mu}\neq0\}=\bigcup_{n=1}^{\infty}[[\tau_n]], \quad
[[\tau_n]]\cap [[\tau_m]]=\emptyset,\quad n\neq m,
\]
where $[[\tau_n]]$ denotes the graph of $\tau_n$.

In what follows, we denote by $A^{\mu,c}$ the continuous part of $A^\mu$
and by $A^{\mu,d}$ the pure jump part of $A^\mu$. For a given c\`adl\`ag
process $Y$,  we write
$\Delta Y_t=Y_t-Y_{t-}\,$, where $Y_{t-}=\lim_{s\nearrow t} Y_s$.

\begin{theorem}
\label{th8.1}
Let $\mu\in\MM_{0,b}(E_{0,T})$ and $(f,g_1,g_2)$ be a
decomposition of $\mu$ from Theorem \ref{th3.2}. Then
\begin{enumerate}
\item[\rm(i)]
$g_2$ has a
quasi-c\`adl\`ag $m_1$-version $\tilde g_2$.
\item[\rm(ii)]
Let $Y=\tilde g_2(X)$. For q.e. $x\in E_{0,T}$ we have
\begin{equation}
\label{eq4.31} A^{\mu,d}_t=\sum_{\tau_n\le t}\Delta Y_{\tau_n} \quad P_x\mbox{-a.s.}
\end{equation}
\end{enumerate}
\end{theorem}
\begin{proof}
Set $\nu=\mu-f\cdot m_1$ and
\[
v(x)=-E_{x}\int^{\zeta}_0dA^{\nu}_t,\quad x\in E_{0,T},
\]
where $A^{\nu}$ is a natural AF of $\BX$ in the Revuz correspondence with
$\nu$. By Proposition \ref{stw.sob15}, $v$ is quasi-c\`adl\`ag and
it is a solution to (\ref{eq7.1}) with $f=0$, $\varphi=0$ and $g$  replaced
by $-\nu$. By (\ref{eq6.5.vir1}),  $g_2=v-w$ $m_1$-a.e. for some $w\in\WW_T$.
%\[
%-\partial_t w-L_tw=-g_1-\chi,\qquad w(T)=0,
%\]
%where $\chi$ is an element of $\VV'$ such that $\langle\chi,\eta\rangle=\BB(g_2,\eta)$, $\eta\in\VV$.
%Since
%$g_1+\chi\in\VV'$, we have $w\in\WW_T$.
Write
\begin{equation}
\label{eq4.19}
\tilde g_2=v-\tilde w.
\end{equation}
By the argument used to prove (\ref{eq.deu}), there is a
martingale AF $M^{\nu}$ of $\BX$ such that
\[
v(X_t)-v(X_0)=M^{\nu}_t+A^{\nu}_t,\quad t\ge 0,\quad
P_{x}\mbox{-a.s.}
\]
for q.e. $x\in E_{0,T}$.  By Fukushima's decomposition
(\ref{eq2.6}),
\[
\widetilde{w}(X_t)-\widetilde{w}(X_0) =M_t^{[w]}+N^{[w]}_t,\quad
t\ge 0,\quad P_{x}\mbox{-a.s.}
\]
for q.e. $x\in E_{0,T}$. Hence
\begin{equation}
\label{eq4.20} \tilde g_2(X_t)-\tilde g_2(X_0)
=M^{\nu}_t-M^{[w]}_t +A^{\nu}_t-N^{[w]}_t,\quad t\ge 0,\quad
P_{x}\mbox{-a.s.}
\end{equation}
for q.e. $x\in E_{0,T}$. Since the filtration $(\FF_t)$ is
quasi-left  continuous, by \cite[Theorems A.3.2, A.3.6]{FOT},  the
martingale AFs $M^{\nu}$ and $M^{[w]}$ admit no predictable jump. Therefore
from (\ref{eq4.20}) and continuity of the functionals $N^{[w]}$
and $A^{f\cdot m_1}$ we get the result.
\end{proof}

\begin{remark}
(i) By (\ref{eq4.19}), $\tilde g_2$ is quasi-continuous if and
only if $v$ is quasi-continuous, which in turn is  equivalent to
the continuity of $A^\mu$.
\smallskip\\
(ii) By \cite[Theorem 16.8]{GS},  there is  a positive
Borel function $h:E_{0,T}\rightarrow\BR$  such that
\[
\Delta A^{\mu,d}_t=h(X_{t-}),\quad t\ge 0,\quad P_x\mbox{-a.s.}
\]
for q.e. $x\in E_{0,T}$. By this and (\ref{eq4.31}),

\[
\tilde g_2(X_{\tau_n})-\tilde g_2(X)_{\tau_n-}
%=\Delta A^{\mu,d}_{\tau_n}
=h(X_{\tau_n-})=h(X_{\tau_n})\quad P_x\mbox{-a.s.},
\]
the second equality being a consequence of the
quasi-continuity of the filtration $(\FF_t)$.
We see that if $\tilde g_2$ is not quasi-continuous,
then the jumps of $\tilde g_2(X)$ which are
not produced by $X$ are produced by $\tilde g_2$. The
size of these jumps is described by $h$.
%Therefore $h$ may be called the ``purely quasi-discontinuous part"
%of $\tilde g_2$.
\end{remark}

\begin{example}
Assume that $B^{(t)}=B^{(0)}$ for $t\in[0,T]$ and there exists  a
strictly positive $\beta:E\rightarrow\BR$ such that $\beta\in
L^1(E;m)\cap C(E)$ (for instance, we may take $E$ to be a bounded
open subset $D$ of $\BR^d$ and consider the classical Dirichlet form
on $L^2(D;dx^0)$ defined as $B^{(t)}(u,v)=\int_D\nabla u(x^0)\nabla
v(x^0)\,dx^0$, $u,v\in H^1_0(D)$). Let $a\in(0,T)$. We  define $\mu$ on
$E_{0,T}$ by
\[
\mu(dt\,dx^0)=(\delta_{\{a\}}+\ell)(dt)\otimes \beta(x^0)\,m(dx^0),
\]
where $\ell$ is the Lebesgue measure on $[0,T]$. Clearly
$\mu\in\MM_b(E_{0,T})$. Moreover, if $\mbox{Cap}_\psi(B)=0$  for
some Borel set $B\subset E_{0,T}$, then $m_1(B)=0$ and
$\mbox{Cap}_\psi(\{t\}\times B_t)=0$ for every $t\in(0,T]$, where
$B_t=\{x^0\in E:(t,x^0)\in B\}$. From this, \cite[(4.4)]{O3} and
\cite[(6.2.24)]{O4} it follows that $m(B_t)$ for $t\in(0,T]$. In
particular, $m(B_a)=0$, and consequently $\mu(B)=0$. Thus
$\mu\in\MM_{0,b}(E_{0,T})$. Let
\[
\alpha(t)=\fch_{[a,T]}(t)+t,\quad g_2(t,x^0)=\alpha(t)\beta(x^0),
\quad (t,x^0)\in E_{0,T}.
\]
Then
\[
\mu=\partial_tg_2.
\]
%(tzn $\mu$ ma rozk/lad $(0,0,g_2)$). Indeed, we have
%\[
%\int_{E_{0,T}}\eta\,d\mu=\int_E\eta(\iota,x)\beta(x)\,m(dx).
%\]
%On the other hand,
%\[
%-\langle\partial_t\eta,g_2\rangle
%=\int^T_0\langle\partial_t\eta(t),\beta\rangle_{V',V}\alpha(t)\,dt = ..
%\]
By (\ref{eq2.8}),
$g_2(X_t)=\alpha(\upsilon(t))\beta(X^0_{\upsilon(t)})$,  $t\ge0$,
$P_x$-a.s.  for q.e. $x\in E_{0,T}$, from which it follows that
$g_2$ is quasi-c\`adl\`ag. Furthermore, under the measure $P_{x}$
with $x=(s,x^0)$, for every predictable $\tau$ we have
\[
\Delta g_2(X)_\tau=(\alpha(s+\tau)-\alpha ((s+\tau)-))\beta(X^0_{s+\tau})
=\begin{cases}
0, &s+\tau\neq a,\\
\beta(X^0_{a}), & s+\tau=a.
\end{cases}
\]
Consequently, $g_2$ is not quasi-continuous. Finally, we note that
\[
\Delta A^{\mu,d}_t=\fch_{\{a\}}(s+t)\beta(X^0_{s+t})=h(X_t), \quad
t\in[0,T-s], \quad P_{s,x^0}\mbox{-a.s.},
\]
where $h(t,x^0)=\fch_{\{a\}}(t)\beta(x^0)$, $(t,x^0)\in E_{0,T}$.
%i.e. $h$ is the purely quasi-discontinuous part of $g_2$.
\end{example}

\subsection*{Acknowledgements}
{\small This work was supported by Polish National Science Centre
(Grant No. 2016/23/B/ST1/01543).}

\end{document}